\documentclass[onecolumn,11pt]{IEEEtran}

\usepackage{macro}

\title{Multiple Packing:\\Lower and Upper Bounds}
\author{
\IEEEauthorblockN{
Yihan Zhang\IEEEauthorrefmark{1}
Shashank Vatedka\IEEEauthorrefmark{2}
}\\
\IEEEauthorblockA{
\IEEEauthorrefmark{1}Institute of Science and Technology Austria \\
\IEEEauthorrefmark{2}Department of Electrical Engineering, Indian Institute of Technology Hyderabad
}
}

\begin{document}
\maketitle

\begin{abstract}
We study the problem of high-dimensional multiple packing in Euclidean space. Multiple packing is a natural generalization of sphere packing and is defined as follows. Let $ N>0 $ and $ L\in\mathbb{Z}_{\ge2} $. A multiple packing is a set $\mathcal{C}$ of points in $ \mathbb{R}^n $ such that any point in $ \mathbb{R}^n $ lies in the intersection of at most $ L-1 $ balls of radius $ \sqrt{nN} $ around points in $ \mathcal{C} $. We study the multiple packing problem for both bounded point sets whose points have norm at most $\sqrt{nP}$ for some constant $P>0$ and unbounded point sets whose points are allowed to be anywhere in $ \mathbb{R}^n $. Given a well-known connection with coding theory, multiple packings can be viewed as the Euclidean analog of list-decodable codes, which are well-studied for finite fields. In this paper, we derive various bounds on the largest possible density of a multiple packing in both bounded and unbounded settings. A related notion called average-radius multiple packing is also studied. Some of our lower bounds exactly pin down the asymptotics of certain ensembles of average-radius list-decodable codes, e.g., (expurgated) Gaussian codes and (expurgated) spherical codes. In particular, our lower bound obtained from spherical codes is the best known lower bound on the optimal multiple packing density and is the first lower bound that approaches the known large $L$ limit under the average-radius notion of multiple packing. To derive these results, we apply tools from high-dimensional geometry and large deviation theory. 
\end{abstract}


\section{Introduction}
\label{sec:intro}
We {study} the problem of \emph{multiple packing} {in Euclidean space}, a natural generalization of the sphere packing problem~\cite{conway-sloane-book}. 
Let $ P,N>0 $ and $ L\in\bZ_{\ge2} $. 
We say that a point set $\cC$ in $ \cB^n(\sqrt{nP}) $\footnote{Here we use $ \cB^n(r) $ to denote an $n$-dimensional Euclidean ball of radius $r$ centered at the origin.} forms a \emph{$(P,N,L-1)$-multiple packing}\footnote{We choose to stick with $L-1$ rather than $L$ for notational convenience. 
This is because in the proof, we need to examine the violation of $(L-1)$-packing, i.e., the existence of an $L$-sized subset that lies in a ball of radius $\sqrt{nN}$. } if any point in $ \bR^n $ lies in the intersection of at most $L-1$ balls of radius $ \sqrt{nN} $ around points in $\cC$. 
Equivalently, the radius of the smallest ball containing any size-$L$ subset of $\cC$ is larger than $ \sqrt{nN} $. 
This radius is known as the \emph{Chebyshev radius} of the $L$-sized subset. 
If $ L = 2 $, then $\cC$ forms a \emph{sphere packing}, i.e., a point set such that balls of radius $ \sqrt{nN} $ around points in $\cC$ are disjoint, or equivalently, the pairwise distance of points in $\cC$ is larger than $ 2\sqrt{nN} $. 
The density of $\cC$ is measured by its \emph{rate} defined as  
\begin{align}
R(\cC) \coloneqq \frac{1}{n}\ln|\cC|. 
\label{eqn:intro-rate}
\end{align}
Denote by $ C_{L-1}(P,N) $ the largest rate of a $(P,N,L-1)$-multiple packing as $ n\to\infty $. 
Note that $ C_{L-1}(P,N) $ depends on $P$ and $N$ only through their ratio $ N/P $ which we call the \emph{noise-to-signal ratio}.
The goal of this paper is to advance the understanding of $ C_{L-1}(P,N) $.

The problem of multiple packing is closely related to the \emph{list-decoding} problem \cite{elias-1957-listdec,wozencraft-1958-listdec} in coding theory. 
Indeed, a multiple packing can be seen exactly as the \emph{Euclidean} analog of a list-decodable code.
We will interchangeably use the terms ``packing'' and ``code'' to refer to the point set of interest. 
To see the connection, note that if any point in a multiple packing is transmitted through an adversarial channel that can inflict an arbitrary additive noise of length at most $ \sqrt{nN} $, then given the distorted transmission, one can decode to a list of the nearest $L-1$ points which is guaranteed to contain the transmitted one. 
The quantity $ C_{L-1}(P,N) $ can therefore be interpreted as the capacity of this channel. 
Moreover, list-decodable codes can be turned into unique-decodable codes with the aid of side information such as common randomness shared between the transmitter and receiver \cite{langberg-focs2004,sarwate-thesis,bhattacharya2019shared}. 
List-decoding also serves as a proof technique towards unique-decoding in various communication scenarios; see, e.g., \cite{zhang-quadratic-arxiv,zhang-2020-twoway}.

Let us start with the $L=2$ case. 
The best known lower bound is due to Blachman in 1962 \cite{blachman-1962} using a simple volume packing argument. 
The best known upper bound is due to Kabatiansky and Levenshtein in 1978 \cite{kabatiansky-1978} using the seminal Delsarte's linear programming framework \cite{delsarte-1973} from coding theory. 
These bounds meet nowhere except at two points $ N/P = 0 $ (where $ C_{L-1}(P,N) = \infty $) and $ N/P = 1/2 $ (where $ C_{L-1}(P,N) = 0 $).

For $ L>2 $, Blinovsky \cite{blinovsky-1999-list-dec-real} claimed a lower bound (\Cref{eqn:compare-lb-ee}) on $ C_{L-1}(P,N) $, but there were some gaps in the proof that we could not resolve.
Part of the original motivation of this work was to obtain intuition and also derive a more rigorous proof of the same result. 
Our approach is completely different from Blinovsky's and we believe that it is conceptually much simpler. 
In fact, our bound holds under a stronger notion known as \emph{average-radius} multiple packing, to be discussed momentarily. 
In the same paper, Blinovsky \cite{blinovsky-1999-list-dec-real} also rederived an upper bound using the ideas of the Plotkin bound \cite{plotkin-1960} and the Elias--Bassalygo bound \cite{bassalygo-pit1965} in coding theory. 
This bound was originally shown by Blachman and Few \cite{blachman-few-1963-multiple-packing} using a more involved approach. 
Blinovsky and Litsyn \cite{blinovsky-litsyn-2011} later improved this bound in the low-rate regime by a recursive application of a bound on the distance distribution by Ben-Haim and Litsyn \cite{benhaim-litsyn-2006-reliability-gaussian}. 
The latter bound in turn relies on the Kabatiansky--Levenshtein linear programming bound \cite{kabatiansky-1978}. 
Blinovsky and Litsyn \cite{blinovsky-litsyn-2011} numerically verified that their bounds improve previous ones when the rate is sufficiently low, but no explicit expression was given.

The above notion of $ (P,N,L-1) $-multiple packing also makes sense if we remove the restriction that all points lie in $ \cB^n(\sqrt{nP}) $ and allow the packing to contain points anywhere in $ \bR^n $. 
This leads to the notion of $ (N,L-1) $-multiple packing. 
The density of an unbounded packing is measured by the (normalized) number of points per volume 
\begin{align}
R(\cC) \coloneqq \limsup_{K\to\infty}\frac{1}{n}\ln\frac{\card{\cC\cap\cB^n(K)}}{\card{\cB^n(K)}}. 
\label{eqn:intro-rate-unbdd}
\end{align}
With slight abuse of terminology, we call $ R(\cC) $ the \emph{rate} of the unbounded packing $\cC$, a.k.a.\ the \emph{normalized logarithmic density (NLD)}. 
The largest density of unbounded multiple packings as $ n\to\infty $ is denoted by $ C_{L-1}(N) $.

For $L=2$, the unbounded sphere packing problem 
has a long history since at least the Kepler conjecture \cite{kepler-1611} in 1611. 
The best known lower bound is due to Minkowski \cite{minkowski-sphere-pack} using a straightforward volume packing argument. 
The best known upper bound is obtained by reducing it to the bounded case for which we have the Kabatiansky--Levenshtein linear programming-type bound \cite{kabatiansky-1978}. 
For $ L>2 $, Blinovsky \cite{blinovsky-2005-random-packing} gave a lower bound by analyzing an (expurgated) Poisson Point Process (PPP). 
Also see \cite{zhang-split-ppp} for a discussion and an alternative proof. 
The paper \cite{blinovsky-2005-random-packing} also presented an Elias--Bassalygo-type bound without a proof. 
We present a proof for it in \Cref{sec:eb} for completeness. 

For the multiple packing problem with $ L>2 $, many existing lower bounds (for both bounded and unbounded settings) are obtained under a stronger notion known as the \emph{average-radius} multiple packing (see \Cref{def:avg-rad-multi-pack-bounded,def:avg-rad-multi-pack-unbounded} for the exact definitions in the bounded and unbounded cases, respectively). 
A set $ \cC $ of $ \bR^n $-valued points is called an average-radius multiple packing if for any $(L-1)$-subset of $\cC$, the maximum distance from any point in the subset to the centroid of the subset is more than $ \sqrt{nN} $. 
Here the centroid of a subset is defined as the average of the points in the subset. 
Denote by $ \ol C_{L-1}(P,N) $ and $ \ol C_{L-1}(N) $ the largest density of average-radius multiple packings in the bounded and unbounded cases, respectively.
In this paper, we treat them as different notions from $ C_{L-1}(P,N) $ and $ C_{L-1}(N) $, and we study both notions. 
For any finite $ L\in\bZ_{\ge2} $, it is unknown whether the largest multiple packing density under the regular notion is the same as that under the average-radius variant. 

For $ L\to\infty $, Zhang and Vatedka \cite{zhang-vatedka-2019-listdecreal} determined the limiting value of $ C_{L-1}(N) $. 
It follows from results in this paper that $ \ol C_{L-1}(N) $ converges to the same value as $ L\to\infty $. 
The limit of $ C_{L-1}(P,N) $ as $ L\to\infty $ is a folklore in the literature and a proof can be found in \cite{zhang-quadratic-arxiv}. 
However, as far as we know, it is not known before this work whether $ \ol C_{L-1}(P,N) $ converges to the same value as $ C_{L-1}(P,N) $. 
It follows from results in this paper (see lower bound in \Cref{thm:lb-spherical-improved} and upper bound in \Cref{thm:eb}) that $ \ol C_{L-1}(P,N) $ does converge to the same value as $ C_{L-1}(P,N) $. 

Very little is known about structured packings. 
Grigorescu and Peikert \cite{grigorescu-peikert-2012-list-dec-barnes-wall} initiated the study of list-decodability of lattices. 
See also the recent work \cite{mook-peikert-2020-lattice} by Mook and Peikert. 
Zhang and Vatedka \cite{zhang-vatedka-2019-listdecreal} had results on list-decodability of random lattices. 

\subsection*{Relation to conference version}
This work was presented in part at the 2022 IEEE International Symposium on Information Theory~\cite{zhang-misc-isit}. 
\cite{zhang-misc-isit} only contains results on Gaussian and spherical codes. 
Besides the these results, the current paper also presents alternative analyses of spherical and ball codes, extends and simplifies the bounds in \cite{blachman-few-1963-multiple-packing}, derives upper bounds on the multiple packing density, and makes various observations regarding bounded/unbounded packing and average radius. 

\section{Related works}
\label{sec:related}
For $L=2$, the problem of (unbounded) sphere packing has a long history and has been extensively studied, especially for small dimensions. 
The largest packing density is open for almost every dimension, except for $ n = 1 $ (trivial), $2$ (\cite{thue1911-2dspherepacking,toth-1940-2dspherepacking}), $ 3 $ (the Kepler conjecture, \cite{hales1998kepler,hales2017formal}), $ 8 $ (\cite{viazovska-2017-8dspherepacking}) and $24$ (\cite{cohn-2017-24spherepacking}). 
For $ n\to\infty $, the best lower and upper bounds remain the trivial sphere packing bound and Kabatiansky--Levenshtein's linear programming bound \cite{kabatiansky-1978}. 
This paper is only concerned with (multiple) packings in high dimensions and we measure the density in the normalized way as mentioned in \Cref{sec:intro}. 

There is a parallel line of research in combinatorial coding theory. 
Specifically, a uniquely-decodable code (resp.\ list-decodable code) is nothing but a sphere packing (resp.\ multiple packing) which has been extensively studied for $ \bF_q^n $ equipped with the Hamming metric. 
Empirically, it seems that the problem is harder for smaller field sizes $q$.

We first list the best known results for sphere packing (i.e., $L=2$) in Hamming spaces. 
For $q=2$, the best lower and upper bounds are the Gilbert--Varshamov bound \cite{gilbert1952,varshamov1957} proved using a trivial volume packing argument and the second MRRW bound \cite{mrrw2} proved using the seminal Delsarte's linear programming framework \cite{delsarte-1973}, respectively. 
Surprisingly, the Gilbert--Varshamov bound can be improved using algebraic geometry codes \cite{goppa1977,tvz} for $ q\ge49 $. 
Note that such a phenomenon is absent in $ \bR^n $; as far as we know, no algebraic constructions of Euclidean sphere packings are known to beat the greedy/random constructions. 
For $ q\ge n $, the largest packing density is known to exactly equal the Singleton bound \cite{komamiya1953singleton,joshi1958singleton,singleton1964} which is met by, for instance, the Reed--Solomon code \cite{reed-solomon}. 

Less is known for multiple packing in Hamming spaces. 
We first discuss the binary case (i.e., $q=2$). 
For every $ L\in\bZ_{\ge2} $, the best lower bound appears to be Blinovsky's bound \cite[Theorem 2, Chapter 2]{blinovsky2012book} proved under the stronger notion of average-radius list-decoding. 
The best upper bound for $L=3$ is due to Ashikhmin, Barg and Litsyn \cite{abl-2000-list-size-2} who combined the MRRW bound \cite{mrrw2} and Litsyn's bound \cite{litsyn-1999} on distance distribution. 
For any $ L\ge4 $, the best upper bound is essentially due to Blinovsky again \cite{blinovsky-1986-ls-lb-binary}, \cite[Theorem 3, Chapter 2]{blinovsky2012book}, though there are some partial improvements. 
In particular, the idea in \cite{abl-2000-list-size-2} was recently generalized to larger $L$ by Polyanskiy \cite{polyanskiy-2016-list-dec} who improved Blinovsky's upper bound for \emph{even} $L$ (i.e., odd $L-1$) and sufficiently large $R$. 
Similar to \cite{abl-2000-list-size-2}, the proof also makes use of a bound on distance distribution due to Kalai and Linial \cite{kalai-linial-1995-distance-distribution} which in turn relies on Delsarte's linear programming bound. 
For larger $q$, Blinovsky's lower and upper bounds \cite{blinovsky-2005-ls-lb-qary,blinovsky-2008-ls-lb-qary-supplementary}, \cite[Chapter III, Lecture 9, \S 1 and 2]{ahlswede-blinovsky-2008-book} remain the best known.

As $L\to\infty$, the limiting value of the largest multiple packing density is a folklore in the literature known as the ``list-decoding capacity'' theorem\footnote{It is an abuse of terminology to use ``list-decoding capacity'' here to refer to the large $L$ limit of the $(L-1)$-list-decoding capacity.}.
Moreover, the limiting value remains the same under the average-radius notion. 

The problem of list-decoding was also studied for settings beyond the Hamming errors, e.g., list-decoding against erasures \cite{guruswami-it2003,ben-aroya-doron-ta-shma-2018-explicit-erasure-ld}, insertions/deletions \cite{guruswami-2020-listdec-insdel}, asymmetric errors \cite{polyanskii-zhang-2021-z}, etc. 
Zhang et al.\ considered list-decoding over general adversarial channels \cite{zhang-2019-list-dec-general}. 
List-decoding against other types of adversaries with \emph{limited} knowledge such as oblivious or myopic adversaries were also considered in the literature \cite{hughes-1997-list-avc,sarwate-gastpar-2012-listdec,zhang-2020-obli-list-dec,hosseinigoki-kosut-2018-oblivious-gaussian-avc-ld,zhang-quadratic-arxiv}. 
The current paper can be viewed as a collection of results for list-decodable codes for adversarial channels over $ \bR $ with $ \ell_2 $ constraints.

\section{Our results}
\label{sec:results}

We derive upper and lower bounds on the largest multiple packing density in both bounded and unbounded settings.
We compare various bounds that appear in this paper. 
Let $ C_{L-1}(P,N) $ and $ \ol C_{L-1}(P,N) $ denote the largest density of multiple packings under the standard and the average-radius notions, respectively.
Let $ C_{L-1}(N) $ and $ \ol C_{L-1}(N) $ denote the largest density of unbounded multiple packings under the standard and the average-radius notions, respectively.

\subsection{{Bounded packings}}
\label{sec:results-bdd}
We first focus on the $(P,N,L-1)$-multiple packing problem. 
Since average-radius multiple packing is stronger than the standard multiple packing, $ C_{L-1}(P,N)\ge \ol C_{L-1}(P,N) $ for any $ L\in\bZ_{\ge2} $ and $ P,N>0 $. 
Therefore any lower bound on $ \ol C_{L-1}(P,N) $ automatically holds for $ C_{L-1}(P,N) $ and any upper bound on $ C_{L-1}(P,N) $ automatically holds for $ \ol C_{L-1}(P,N) $. 
Similar statements also hold for $ C_{L-1}(N) $ vs.\ $ \ol C_{L-1}(N) $.

In \Cref{thm:lb-gaussian}, we obtain the \emph{exact} asymptotics of (expurgated) Gaussian codes under $(P,N,L-1)$-average-radius list-decoding which serves as a lower bound on $ \ol C_{L-1}(P,N) $:
\begin{align}
\ol C_{L-1}(P,N) &\ge \frac{1}{2}\sqrbrkt{\ln\frac{(L-1)P}{LN} + \frac{LN}{(L-1)P} + 1}. \label{eqn:compare-lb-gaussian} 
\end{align}
In \Cref{thm:lb-spherical}, we obtain a lower bound on the rate that can be achieved by (expurgated) spherical codes under $ (P,N,L-1) $-average-radius list-decoding: 
\begin{align}
\ol C_{L-1}(P,N) &\ge \frac{1}{2}\sqrbrkt{1 - \frac{LN}{(L-1)P} + \frac{1}{L-1}\ln\frac{P}{L(P-N)}}. \label{eqn:compare-lb-spherical} 
\end{align}
Using refined analysis, we then show in \Cref{thm:lb-spherical-improved} that (expurgated) spherical codes in fact achieve a much higher rate under average-radius list-decoding:
\begin{align}
\ol C_{L-1}(P,N) &\ge \frac{1}{2}\sqrbrkt{\ln\frac{(L-1)P}{LN} + \frac{1}{L-1}\ln\frac{P}{L(P-N)}}. \label{eqn:compare-lb-spherical-improved}
\end{align}
The following bound was initially proved by Blachman and Few \cite{blachman-few-1963-multiple-packing} for $ C_{L-1}(P,N) $.
In \Cref{thm:lb-blachman-few}, we simplify their proof and strengthen it so that it holds for $ \ol C_{L-1}(P,N) $: 
\begin{align}
\ol C_{L-1}(P,N) &\ge \frac{1}{2}\ln\frac{(L-1)^2P^2}{LN(2(L-1)P - LN)}. \label{eqn:compare-lb-blachman-few} 
\end{align}
The idea due to Blachman and Few is to reduce the multiple packing problem to sphere packing.
The following is the best known lower bound on the $(P,N,L-1)$-list-decoding capacity
\begin{align}
C_{L-1}(P,N) &\ge \frac{1}{2}\sqrbrkt{\ln\frac{(L-1)P}{LN} + \frac{1}{L-1}\ln\frac{P}{L(P-N)}}. \label{eqn:compare-lb-ee}
\end{align}
The above bound was claimed by Blinovsky \cite{blinovsky-1999-list-dec-real} (also see~\cite{zhang-split-ee}).
Curiously, we observe that \Cref{eqn:compare-lb-ee,eqn:compare-lb-spherical-improved} coincide, though \Cref{eqn:compare-lb-spherical-improved} is proved under a stronger notion \emph{average-radius} list-decoding and the underlying proof technique is rather distinct from that for \Cref{eqn:compare-lb-ee}. 
In the same paper, Blinovsky also claimed the following upper bound on $ C_{L-1}(P,N) $:
\begin{align}
C_{L-1}(P,N) &\le \frac{1}{2}\ln\frac{(L-1)P}{LN}. \label{eqn:compare-eb} 
\end{align}
We make Blinovsky's arguments for the above bound fully rigorous and complete and present a cleaner proof in \Cref{thm:eb}. 
Finally, it is a folklore (whose proof can be found in \cite{zhang-quadratic-arxiv}) that as $ L\to\infty $, $ C_{L-1}(P,N) $ converges to the following expression:
\begin{align}
C_{\mathrm{LD}}(P,N) &= \frac{1}{2}\ln\frac{P}{N}. \label{eqn:compare-ld-cap} 
\end{align}
In fact, our lower bound in \Cref{eqn:compare-lb-spherical-improved} shows that $ \ol{C}_{L-1}(P,N) $ (which is always at most $ C_{L-1}(P,N) $) converges to a limit $ \ol{C}_{\mathrm{LD}}(P,N) $ with the same value $ \frac{1}{2}\ln\frac{P}{N} $ as $L\to\infty$.

Our upper and lower bounds are nowhere tight except for $ N = \frac{L-1}{L}P $ at which point both the upper and lower bounds vanish. 
We refer to this threshold $ \frac{N}{P} = \frac{L-1}{L} $ as the \emph{Plotkin point}. 
As $ L\to\infty $, the Plotkin point increases from $ 1/2 $ to $ 1 $.

All the above bounds for $ (P,N,L-1) $-multiple packing are plotted in \Cref{fig:ld-bdd} with $ L = 5 $. 
The horizontal axis is the noise-to-signal ratio $ N/P $ and the vertical axis is the value of various bounds. 
The largest lower bound turns out to be \Cref{eqn:compare-lb-ee,eqn:compare-lb-spherical-improved} (for all $N,P\ge0$ and $ L\in\bZ_{\ge2} $). 
We also plot the best lower bound \Cref{eqn:compare-lb-spherical-improved} together with the Elias--Bassalygo-type upper bound \Cref{eqn:compare-eb} in \Cref{fig:EEEB} for $ L=3,4,5 $. 
They both converge from below to \Cref{eqn:compare-ld-cap} as $ L $ increases. 

\begin{figure}[htbp]
	\centering
	\begin{subfigure}[b]{\linewidth}
		\centering
		\includegraphics[width=0.95\textwidth]{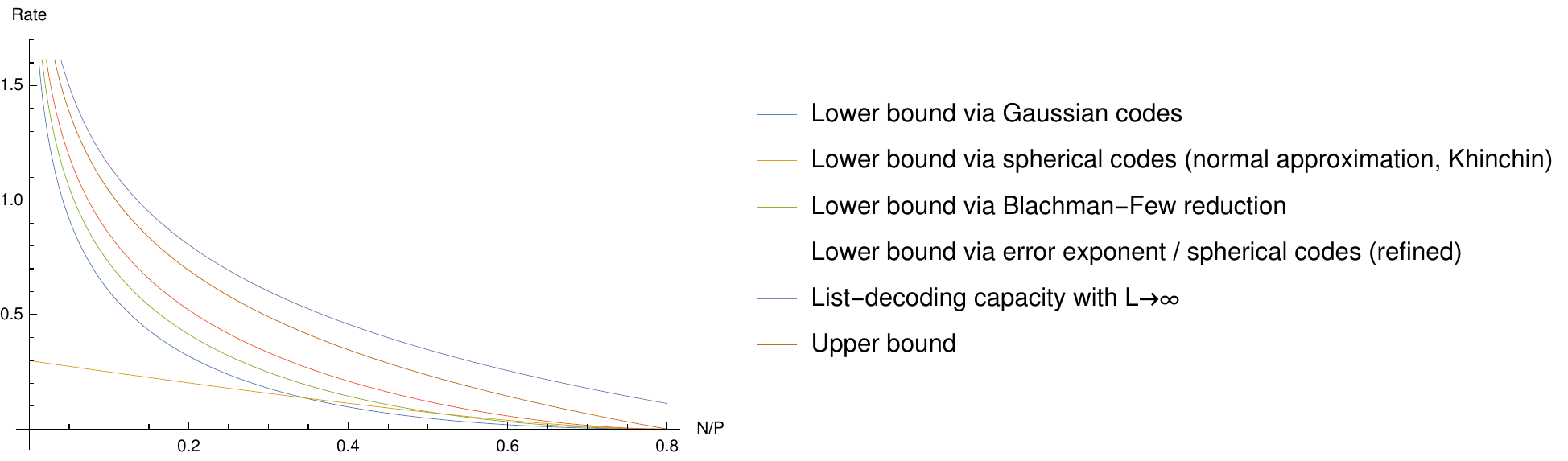}
		\caption{}
		\label{fig:compare}
	\end{subfigure}
	\\
	\begin{subfigure}[b]{\linewidth}
		\centering
		\includegraphics[width=0.95\textwidth]{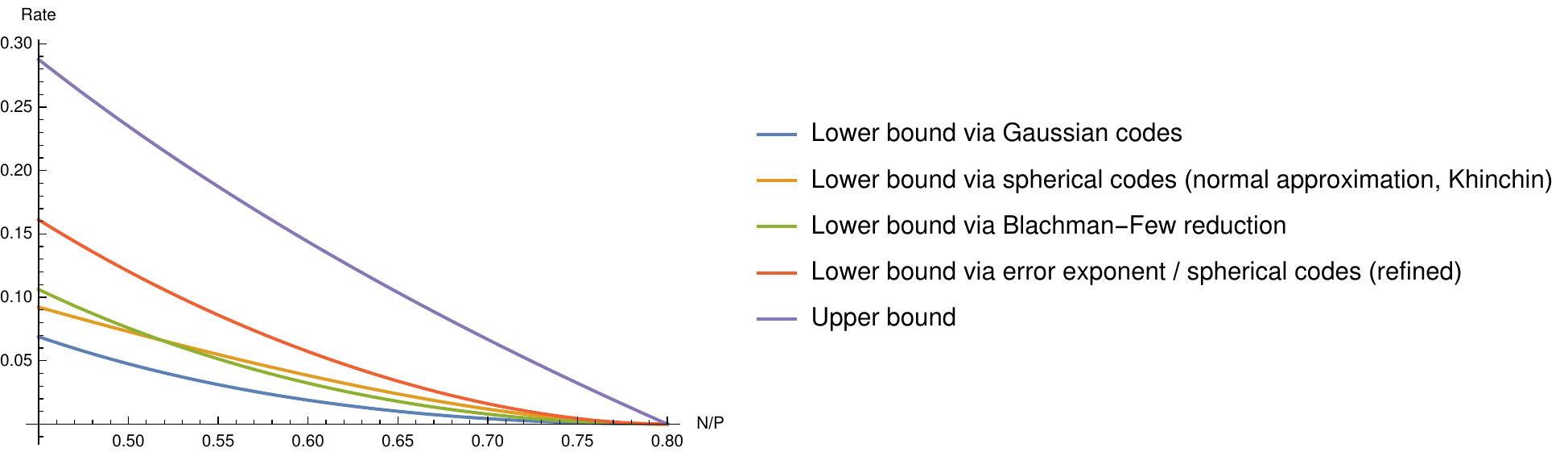}
		\caption{}
		\label{fig:compare-low}
	\end{subfigure}
	\\
	\begin{subfigure}[b]{\linewidth}
		\centering
		\includegraphics[width=0.95\textwidth]{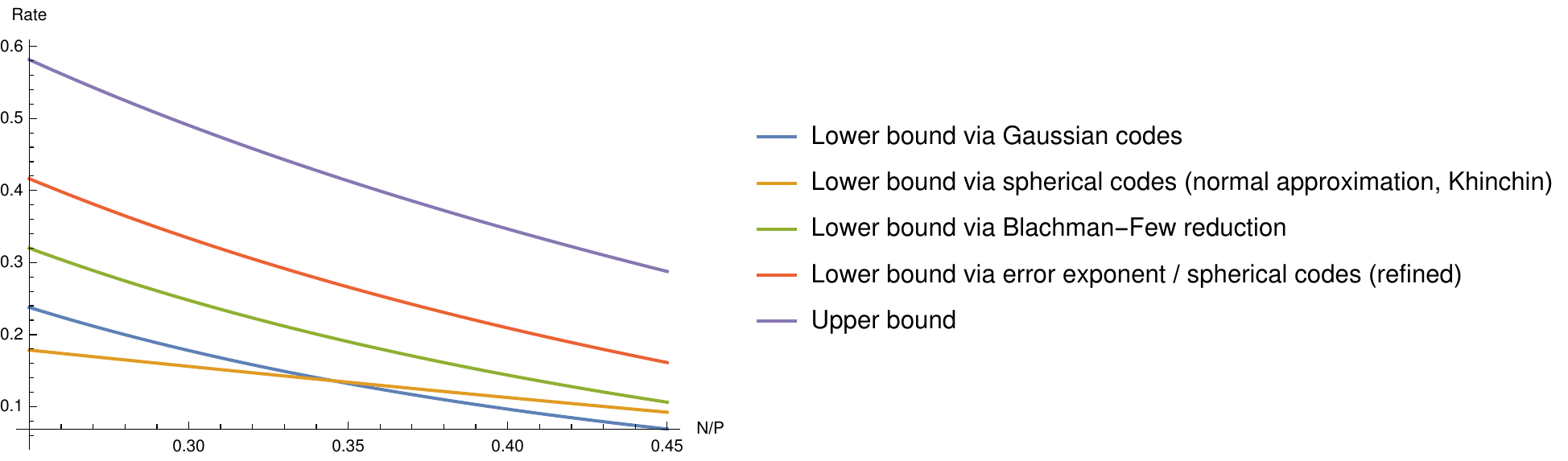}
		\caption{}
		\label{fig:compare-mid}
	\end{subfigure}
	\caption{Comparison of different bounds for the $(P,N,L-1)$-list-decoding problem. The horizontal axis is $ N/P $ and the vertical axis is the value of various bounds. Recall that the rate (\Cref{eqn:intro-rate}) of a bounded packing is defined as its normalized cardinality. We plot bounds for $ L = 5 $. The expressions for all bounds in \Cref{fig:compare} can be found in \Cref{eqn:compare-lb-gaussian,eqn:compare-lb-spherical,eqn:compare-lb-blachman-few,eqn:compare-lb-ee,eqn:compare-eb,eqn:compare-ld-cap,eqn:compare-lb-spherical-improved}. These bounds intersect with each other in a sophisticated manner. Therefore, we zoom in on the low rate regime in \Cref{fig:compare-low} and on the middle rate regime regime in \Cref{fig:compare-mid}. }
	\label{fig:ld-bdd}
\end{figure}

\begin{figure}[htbp]
	\centering
	\begin{subfigure}[b]{\linewidth}
		\centering
		\includegraphics[width=0.95\textwidth]{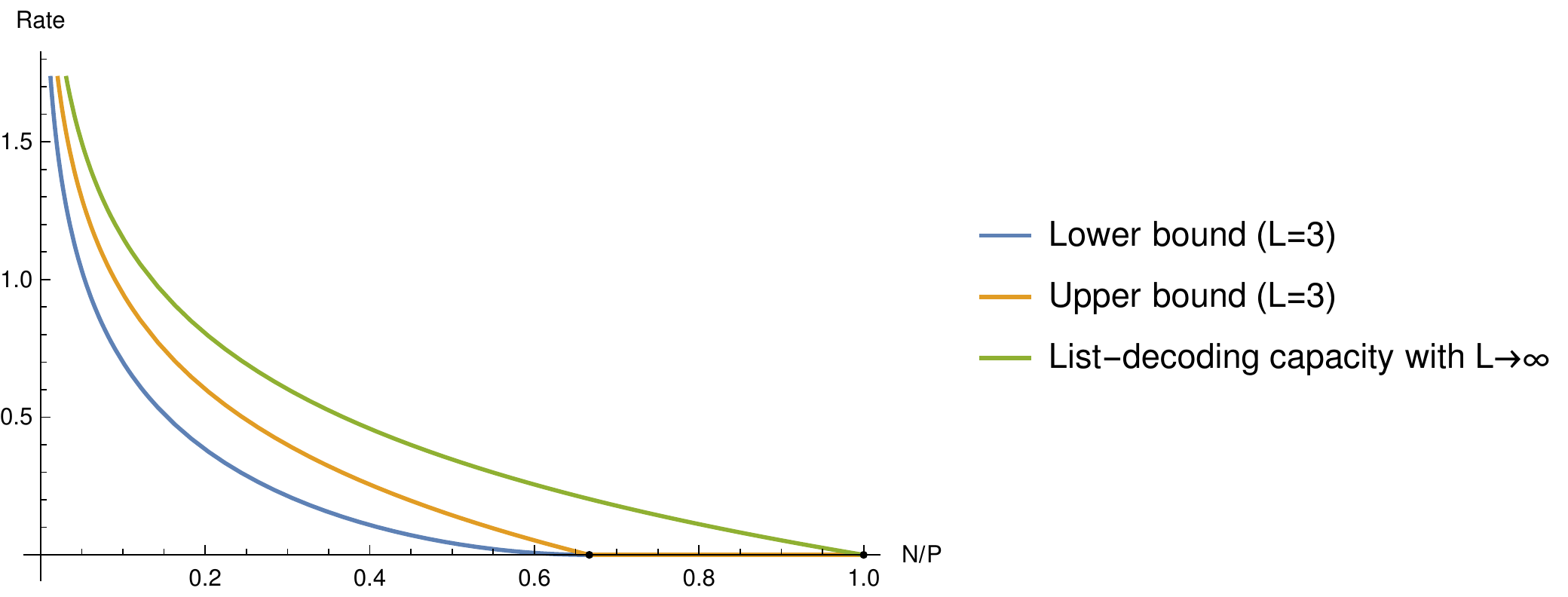}
		\caption{}
		\label{fig:EEEB3}
	\end{subfigure}
	\\
	\begin{subfigure}[b]{\linewidth}
		\centering
		\includegraphics[width=0.95\textwidth]{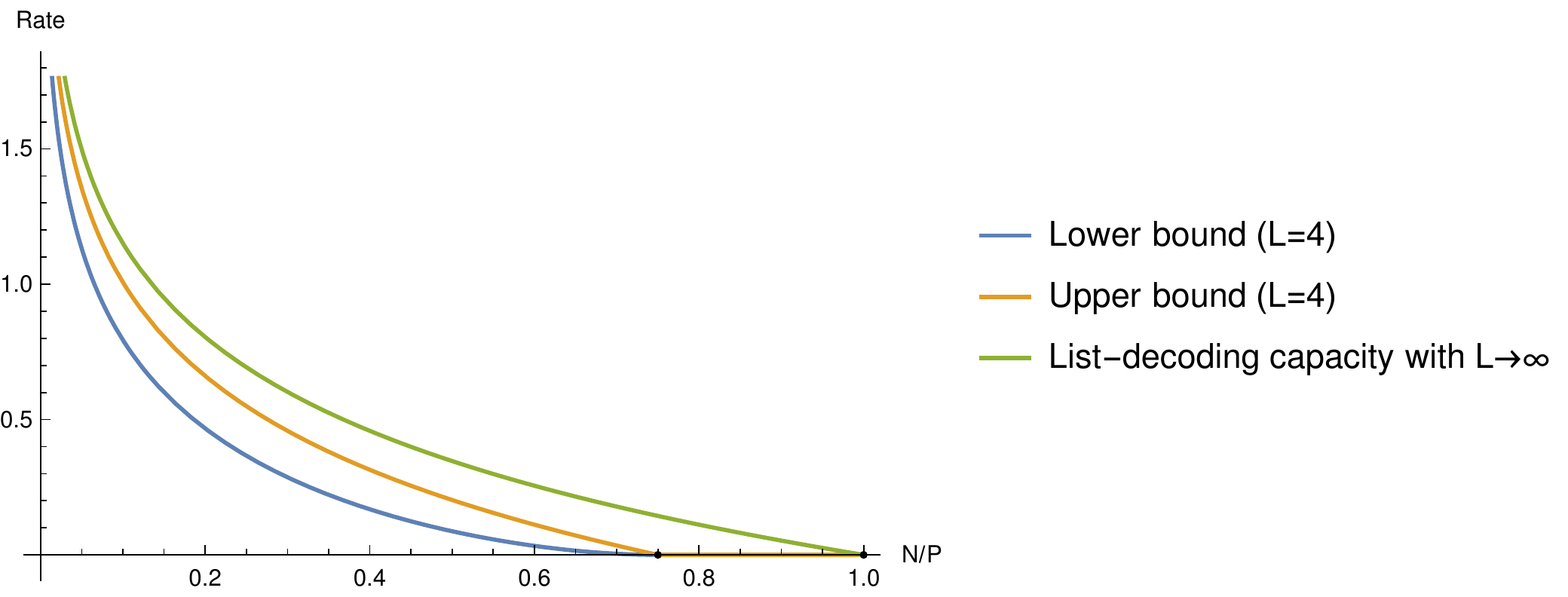}
		\caption{}
		\label{fig:EEEB4}
	\end{subfigure}
	\\
	\begin{subfigure}[b]{\linewidth}
		\centering
		\includegraphics[width=0.95\textwidth]{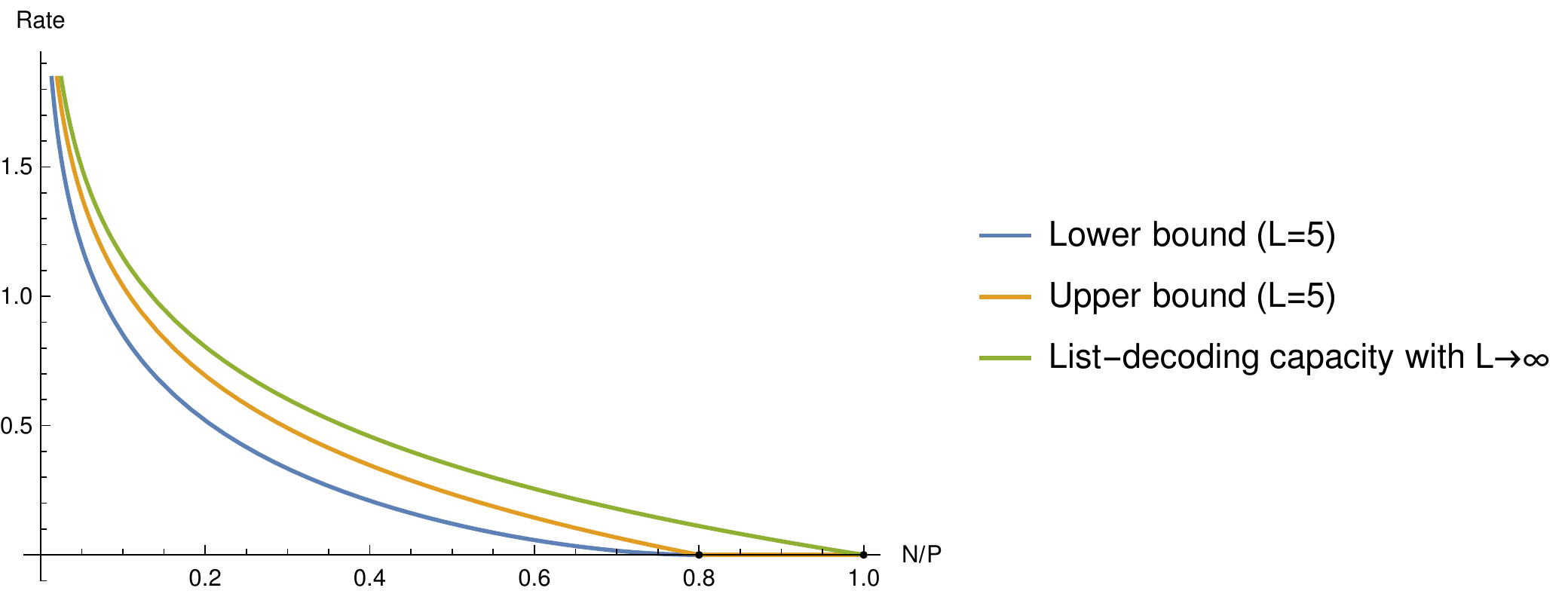}
		\caption{}
		\label{fig:EEEB5}
	\end{subfigure}
	\caption{Plots of the best known lower bound (\Cref{eqn:compare-lb-ee}) on $ C_{L-1}(P,N) $ and the Elias--Bassalygo-type upper bound (\Cref{eqn:compare-eb}) for $ L=3,4,5 $. As $L$ increases, they both converge from below to $ C_{\mathrm{LD}}(P,N) $ (\Cref{eqn:compare-ld-cap}). 
    }
	\label{fig:EEEB}
\end{figure}

\subsection{{Unbounded packings}}
\label{sec:results-unbdd}
We then juxtapose various bounds for the $(N,L-1)$-multiple packing problem. 
In \cite{zhang-split-ee}, the following lower bound (\Cref{eqn:compare-lb-ppp}) on $ C_{L-1}(N) $ is obtained via a connection with error exponents.
Furthermore, in \cite{zhang-split-ppp}, it is shown that the same bound can be achieved by a certain ensemble of infinite constellations under $(N,L-1)$-\emph{average-radius} list-decoding (which is stronger than $ (N,L-1) $-list-decoding): 
\begin{align}
C_{L-1}(N) \ge 
\ol C_{L-1}(N) &\ge \frac{1}{2}\ln\frac{L-1}{2\pi eNL} - \frac{\ln L}{2(L-1)}. \label{eqn:compare-lb-ppp}
\end{align}
In \cite{blinovsky-2005-random-packing}, Blinovsky considered PPPs and arrived at the same bound. 
Also see \cite{zhang-split-ppp} for a discussion of the relation between the two works.
The techniques in \cite{blachman-few-1963-multiple-packing} can be adapted to the unbounded setting and be strengthened to work for the stronger notion of \emph{average-radius} multiple packing. 
They yield the following lower bound on $ \ol C_{L-1}(N) $: 
\begin{align}
\ol C_{L-1}(N) &\ge \frac{1}{2}\ln\frac{L-1}{4\pi eNL}. \label{eqn:compare-lb-bf-unbdd} 
\end{align}
As for upper bound, the techniques in \Cref{thm:eb} can be adapted to the unbounded setting as well which yield the following upper bound on $ C_{L-1}(N) $: 
\begin{align}
C_{L-1}(N) &\le \frac{1}{2}\ln\frac{L-1}{2\pi eNL}. \label{eqn:compare-eb-unbdd}
\end{align}
Finally, it is known (see, e.g., \cite{zhang-vatedka-2019-listdecreal}) that as $ L\to\infty $, $ C_{L-1}(N) $ converges to the following expression: 
\begin{align}
C_{\mathrm{LD}}(N) &= \frac{1}{2}\ln\frac{1}{2\pi eN}. \label{eqn:compare-ld-cap-unbdd}
\end{align}
Note that, by the lower and upper bounds (\Cref{eqn:compare-lb-ppp,eqn:compare-eb-unbdd}) on $ \ol C_{L-1}(N) $ for finite $L$, the limiting value of $ \ol C_{L-1}(N) $ as $ L\to\infty $ is also the above expression. 

All the above bounds for $ (N,L-1) $-multiple packing are plotted in \Cref{fig:ld-unbdd} with $ L = 5 $. 
The horizontal axis is $N$ and the vertical axis is the value of various bounds. 
The largest lower bound turns out to be \Cref{eqn:compare-lb-ppp} (for all $N\ge0$ and $ L\in\bZ_{\ge2} $).
This bound together with the Elias--Bassalygo-type upper bound \Cref{eqn:compare-eb} are plotted in \Cref{fig:PPPEB} for $ L=3,4,5 $. They both converge from below to \Cref{eqn:compare-ld-cap-unbdd} as $ L $ increases.

\begin{figure}[htbp]
	\centering
	\includegraphics[width=0.95\textwidth]{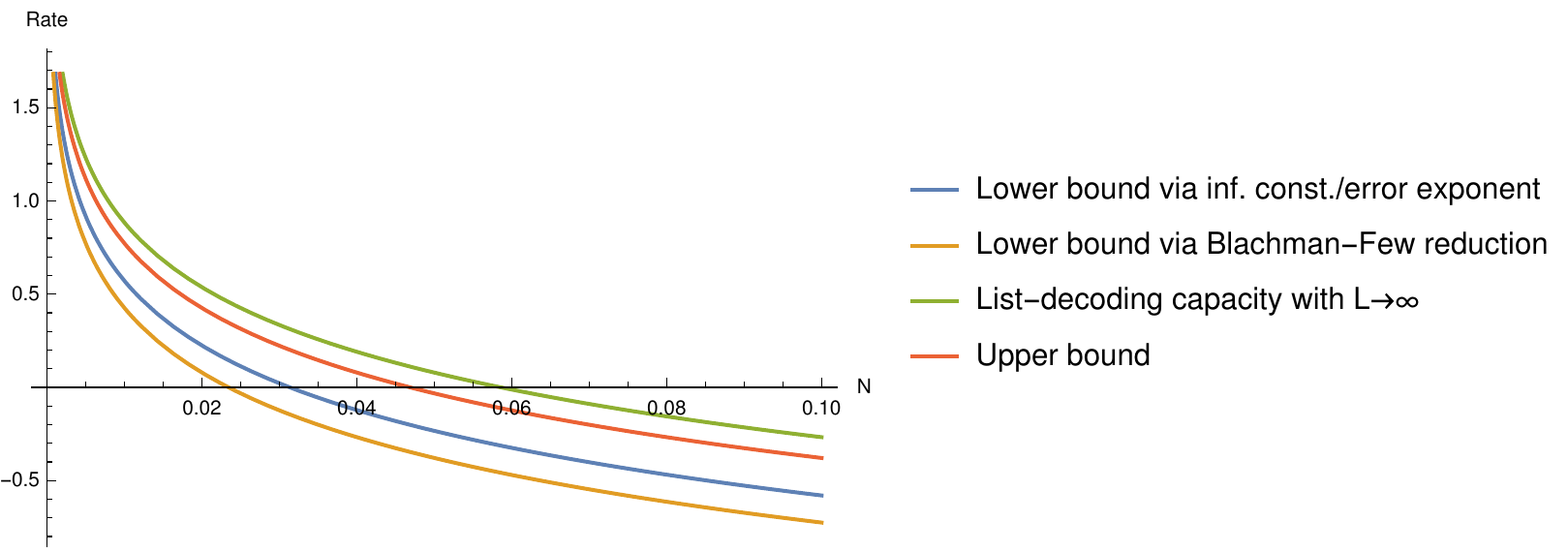}
	\caption{Comparison of different bounds for the $ (N,L-1) $-list-decoding problem. The horizontal axis is $N$ and the vertical axis is the value of bounds. We plot bounds for $ L = 5 $. Recall that the rate (\Cref{eqn:intro-rate-unbdd}) of an unbounded packing is defined as the (normalized) number of points per volume which can be negative. The expressions for all bounds can be found in \Cref{eqn:compare-lb-ppp,eqn:compare-lb-bf-unbdd,eqn:compare-eb-unbdd,eqn:compare-ld-cap-unbdd}. }
	\label{fig:ld-unbdd}
\end{figure}

\begin{figure}[htbp]
	\centering
	\includegraphics[width=0.95\textwidth]{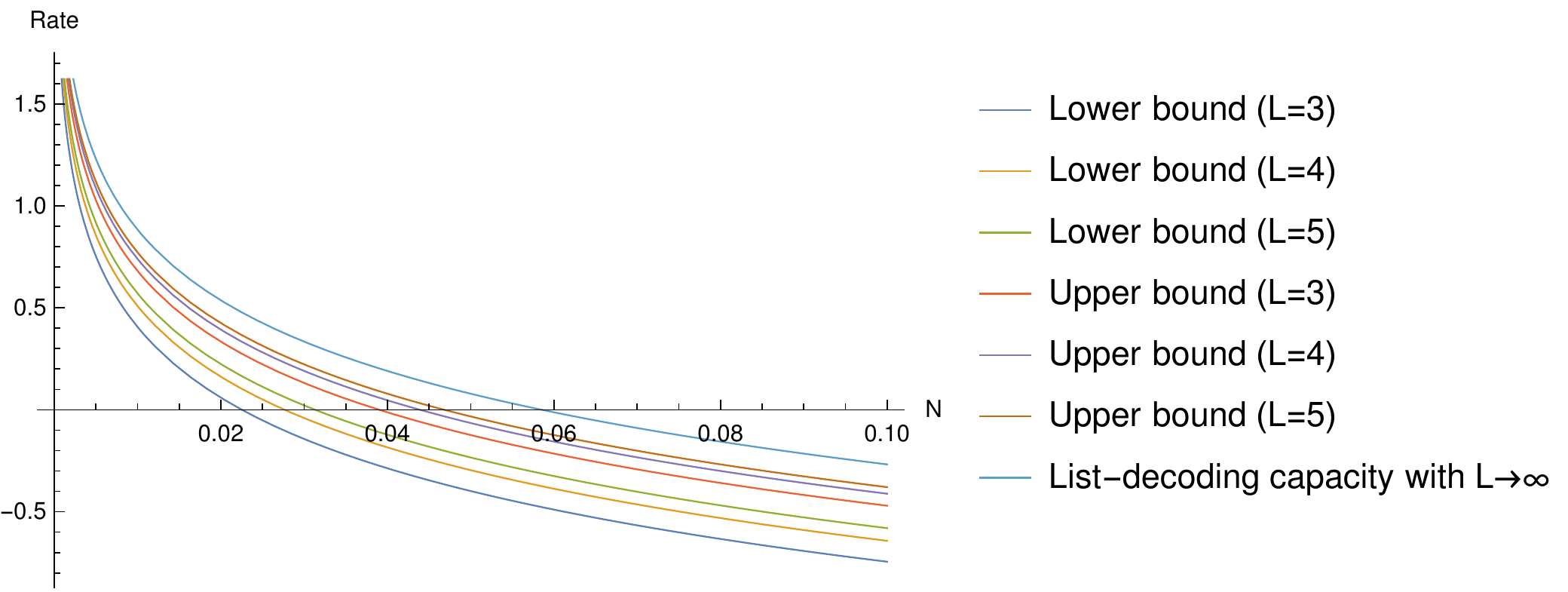}
	\caption{Plots of the best known lower bound (\Cref{eqn:compare-lb-ppp}) on $ C_{L-1}(N) $ and the Elias--Bassalygo-type upper bound for $ L=3,4,5 $. As $ L $ increases, they both converge from below to $ C_{\mathrm{LD}}(N) $ (\Cref{eqn:compare-ld-cap-unbdd}). The lower bound \Cref{eqn:compare-lb-ppp} can be obtained using the connection with error exponents. Moreover, it is actually the exact asymptotics of a certain ensemble of infinite constellations under the \emph{average-radius} notion of unbounded multiple packing. }
	\label{fig:PPPEB}
\end{figure}

\section{List-decoding capacity for large $L$}
\label{sec:listdec-cap-large}

All bounds in this paper hold for any \emph{fixed} $L$. 
In this section, we discuss the impact of our finite-$L$ bounds on the understanding of the limiting values of the largest multiple packing density as $ L\to\infty $. 
Some of these results were known previously and others follow from the bounds in the current paper.

Characterizing $ C_{L-1}(P,N) $, $ \ol C_{L-1}(P,N) $, $ C_{L-1}(N) $ or $ \ol C_{L-1}(N) $ is a difficult task that is out of reach given the current techniques. 
However, if the list-size $L$ is allowed to grow, we can actually characterize 
\begin{align}
C_{\mathrm{LD}}(P,N) &\coloneqq \lim_{L\to\infty} C_{L-1}(P,N), \quad
C_{\mathrm{LD}}(N) \coloneqq \lim_{L\to\infty} C_{L-1}(N), \notag \\
\ol C_{\mathrm{LD}}(N) &\coloneqq \lim_{L\to\infty} \ol C_{L-1}(N), \quad
\ol C_{\mathrm{LD}}(P,N) \coloneqq \lim_{L\to\infty} \ol C_{L-1}(P,N). \notag 
\end{align}
where the subscript $ \mathrm{LD} $ denotes List-Decoding. 

It is well-known that $ C_{\mathrm{LD}}(P,N) = \frac{1}{2}\ln\frac{P}{N} $. 
Specifically, the following theorem appears to be a folklore in the literature and a complete proof can be found in \cite{zhang-quadratic-arxiv}. 
\begin{theorem}[Folklore, \cite{zhang-quadratic-arxiv}]
\label{thm:listdec-cap-large-bdd}
Let $ 0<N\le P $. 
Then for any $ \eps>0 $, 
\begin{enumerate}
	\item There exist $ (P,N,L-1) $-multiple packings of rate $ \frac{1}{2}\ln\frac{P}{N} - \eps $ for some $ L = \cO\paren{\frac{1}{\eps}\ln\frac{1}{\eps}} $; 
	
	\item Any $ (P,N,L-1) $-multiple packing of rate $ \frac{1}{2}\ln\frac{P}{N} + \eps $ must satisfy $ L = e^{\Omega(n\eps)} $. 
\end{enumerate}
Therefore, $ C_{\mathrm{LD}}(P,N) = \frac{1}{2}\ln\frac{P}{N} $. 
\end{theorem}

We claim that $ \ol{C}_{\mathrm{LD}}(P,N) $ is also equal to $ \frac{1}{2}\ln\frac{P}{N} $. 
For an upper bound, recall that average-radius list-decodability implies (regular) list-decodability. 
Therefore, any upper bound on $ C_{L-1}(P,N) $ is also an upper bound on $ \ol C_{L-1}(P,N) $. 
We already saw an upper bound on $ C_{L-1}(P,N) $ in \Cref{thm:eb} that approaches $ \frac{1}{2}\ln\frac{P}{N} $ as $L\to\infty$. 
On the other hand, we have seen constructions of $ (P,N,L-1) $-average-radius list-decodable codes that attain rates approaching $ \frac{1}{2}\ln\frac{1}{2\pi eN} $ as $ L\to\infty $. 
Indeed, according to \Cref{thm:lb-spherical-improved}, for sufficiently large $L$, an expurgated spherical code achieves $ \frac{1}{2}\ln\frac{P}{N} $ with high probability. 
This leads to the following characterization

\begin{theorem}
\label{thm:avgrad-listdec-cap-large-bdd}
For any $P,N>0$, $ \ol{C}_{\mathrm{LD}}(P,N) = \frac{1}{2}\ln\frac{P}{N} $. 
\end{theorem}

The unbounded version $ C_{\mathrm{LD}}(N) $ is characterized in \cite{zhang-vatedka-2019-listdecreal} which equals $ \frac{1}{2}\ln\frac{1}{2\pi eN} $. 
\begin{theorem}[\cite{zhang-vatedka-2019-listdecreal}]
\label{thm:listdec-cap-large-unbdd}
Let $ N>0 $. 
Then for any $ \eps>0 $, 
\begin{enumerate}
	\item There exist $ (N,L-1) $-multiple packings of rate $ \frac{1}{2}\ln\frac{1}{2\pi eN} - \eps $ for some $ L = \cO\paren{\frac{1}{\eps}\ln\frac{1}{\eps}} $; 
	
	\item Any $ (N,L-1) $-multiple packing of rate $ \frac{1}{2}\ln\frac{1}{2\pi eN} + \eps $ must satisfy $ L = e^{\Omega(n\eps)} $. 
\end{enumerate}
Therefore, $ C_{\mathrm{LD}}(N) = \frac{1}{2}\ln\frac{1}{2\pi eN} $. 
\end{theorem}

Finally, we claim $ \ol C_{\mathrm{LD}}(N) = \frac{1}{2}\ln\frac{1}{2\pi eN} $. 
Indeed, \Cref{thm:eb-unbdd} provides an upper bound that approaches $ \frac{1}{2}\ln\frac{1}{2\pi eN} $ as $L\to\infty$. 
On the other hand, according to \cite{zhang-split-ppp}, for sufficiently large $L$, an infinite constellation drawn from a certain ensemble achieves $ \frac{1}{2}\ln\frac{1}{2\pi eN} $ with high probability. 

\begin{theorem}
\label{thm:avgrad-listdec-cap-large-unbdd}
For any $ N>0 $, $ \ol C_{\mathrm{LD}}(N) = \frac{1}{2}\ln\frac{1}{2\pi eN} $. 
\end{theorem}

\section{Our techniques}
\label{sec:techniques}

We summarize our techniques below. 

For lower bounds, our basic strategy is random coding with expurgation, a standard tool from information theory. 
To show the existence of a list-decodable code of rate $R$, we simply randomly sample $ e^{nR} $ points independently each according to a certain distribution. 
There might be bad lists of size $L$ that violates the multiple packing condition. 
We then throw away (a.k.a.\ expurgate) one point from each of the bad lists. 
By carefully analyzing the error event and choosing a proper rate, we can guarantee that the remaining code has essentially the same rate after the removal process. 
We then get a list-decodable code of rate $R$ by noting that the remaining code contains no bad lists.

In the above framework, the key ingredient is a good estimate on the probability of the error event, i.e., the probability that the \emph{list-decoding radius} of a size-$L$ list is smaller than $ \sqrt{nN} $. 
Under the standard notion of multiple packing, the list-decoding radius is the \emph{Chebyshev radius} of the list, i.e., the radius of the smallest ball containing the list. 
Under the average-radius notion of multiple packing, the list-decoding radius is the \emph{average radius} of the list, i.e., the square root of the average squared distance from each point in the list to the centroid of the list. 
Due to the analytic simplicity of the average-radius variant, we manage to obtain the \emph{exact} rate of certain ensembles of average-radius list-decodable codes. 
In particular, we get the exact rates of expurgated Gaussian codes and expurgated spherical codes in the bounded setting. 
The exact exponents of the probability of the error event are obtained using \cramer's large deviation principle and  Laplace's method. 
Leveraging tools from high-dimensional probability including Khinchin's inequality, we also have bounds for other ensembles of codes such as (expurgated) ball codes.

The upper bounds in this paper are neither novel nor sharp. 
They were sketched in Blinovsky's works \cite{blinovsky-1999-list-dec-real,blinovsky-2005-random-packing} and were improved in followup works \cite{blinovsky-litsyn-2011}. 
The purpose of presenting them here is to provide a unified exposition and to clean up the literature. 
We close some of the gaps in the analysis of~\cite{blinovsky-1999-list-dec-real,blinovsky-2005-random-packing,blinovsky-litsyn-2011} and also present some new techniques. 
As an ingredient for the Elias--Bassalygo-type bound, an explicit Plotkin-type bound on the optimal code size above the so-called \emph{Plotkin point} is presented. 
Specifically, it is known (and also follows from our bounds) that as $ n\to\infty $, the largest $(P,N,L-1)$-multiple packing has size $ \cO(1) $ as long as $ N>\frac{L-1}{L}P $. 
The threshold $ \frac{N}{P} = \frac{L-1}{P} $ is called the Plotkin point. 
We present bounds on the optimal code size as a function of $ \rho>0 $ when $ N = \frac{L-1}{L}P(1 + \rho) $. 
Our bound reads as $ F(\rho^{-1} + 1) $ where $ F(\cdot) $ scales like the Ramsey number. 
This bound is far from being tight and indeed it was recently improved by Alon, Bukh and Polyanskiy \cite{abp-2018}. 
The exact dependence on $ \rho $ of the optimal size of codes above the Plotkin point is open. 
The suboptimal dependence of our Plotkin-type bound on $\rho$ turns out to be sufficient for the proof of the Elias--Bassalygo-type bound.

As another technical contribution, we discover several new representations of the average radius and the Chebyshev radius. 
They play crucial roles in facilitating the analyses and the applications of some of these representations go beyond the scope of this paper. 
We show that the average squared radius of a list can be written as a quadratic form associated with the list. 
This representation is used to analyze Gaussian codes and spherical codes in this paper and infinite constellations in \cite{zhang-split-ppp}. 
The average squared radius can also be written as the difference between the average norm of points in the list and the norm of the centroid of the list. 
This representation is used to analyze spherical codes and ball codes. 
The average squared radius can be further written as the average pairwise distance of the list. 
This allows us to give a one-line proof of the Blachman--Few reduction and its strengthened version. 
Yet another way of writing the average squared radius using the average norm and the average pairwise correlation turns out to be useful for the proof of the Plotkin-type bound.

\section{Organization of the paper}
\label{sec:org}

This paper is a collection of lower and upper bounds on the largest multiple packing density. 
The rest of the paper is organized as follows. 
Notational conventions and preliminary definitions/facts are listed in \Cref{sec:notation,sec:prelim}, respectively. 
After that, we present in \Cref{sec:def} the formal definitions of multiple packing and pertaining notions. 
We also discuss different notions of density of codes used in the literature. 
Furthermore, we obtain several novel representations of the Chebyshev radius and the average radius which are crucial for estimating their tail probabilities. 
In \Cref{sec:rand-cod-exp}, we introduce the basic technology for proving lower bounds: random coding with expurgation.

The rest of the sections are relatively independent of each other and can be read in an arbitrary order.
In \Cref{sec:lb-gaussian}, we analyze the exact asymptotics of expurgated Gaussian codes under average-radius list-decoding. 
In \Cref{sec:lb-spherical}, we obtain the exact rate of expurgated spherical codes and bounds for expurgated ball codes, both under average-radius list-decoding. 
In \Cref{sec:lb-blachman-few}, we rederive, simplify and strengthen a lower bound by Blachman and Few \cite{blachman-few-1963-multiple-packing} using a novel representation of average radius in \cite{zhang-split-ppp}. 
Proofs of the Plotkin-type bound and the Elias--Bassalygo-type bound are presented in \Cref{sec:ub}. 
We end the paper with several open questions in \Cref{sec:open}.

\section{Notation}
\label{sec:notation}
\noindent\textbf{Conventions.}
Sets are denoted by capital letters in calligraphic typeface, e.g., $\cC,\cB$, etc. 
Random variables are denoted by lower case letters in boldface or capital letters in plain typeface, e.g., $\bfx,S$, etc. Their realizations are denoted by corresponding lower case letters in plain typeface, e.g., $x,s$, etc. Vectors (random or fixed) of length $n$, where $n$ is the blocklength without further specification, are denoted by lower case letters with  underlines, e.g., $\vbfx,\vbfg,\vx,\vg$, etc. 
Vectors of length different from $n$ are denoted by an arrow on top and the length will be specified whenever used, e.g., $ \vec t, \vec\alpha $, etc. 
The $i$-th entry of a vector $\vx\in\cX^n$ is denoted by $\vx(i)$ since we can alternatively think of $\vx$ as a function from $[n]$ to $\cX$. Same for a random vector $\vbfx$. Matrices are denoted by capital letters, e.g., $A,\Sigma$, etc. Similarly, the $(i,j)$-th entry of a matrix $G\in\bF^{n\times m}$ is denoted by $G(i,j)$. We sometimes write $G_{n\times m}$ to explicitly specify its dimension. For square matrices, we write $G_n$ for short. Letter $I$ is reserved for identity matrix.  

\noindent\textbf{Functions.}
We use the standard Bachmann--Landau (Big-Oh) notation for asymptotics of real-valued functions in positive integers. 

For two real-valued functions $f(n),g(n)$ of positive integers, we say that $f(n)$ \emph{asymptotically equals} $g(n)$, denoted $f(n)\asymp g(n)$, if 
\[\lim_{n\to\infty}\frac{f(n)}{g(n)} = 1.\]
For instance, $2^{n+\log n}\asymp2^{n+\log n}+2^n$, $2^{n+\log n}\not\asymp2^n$.
 We write $f(n)\doteq g(n)$ (read $f(n)$ dot equals $g(n)$) if the coefficients of the dominant terms in the exponents of $f(n)$ and $g(n)$ match,
\[\lim_{n\to\infty}\frac{\log f(n)}{\log g(n)} = 1.\]
For instance, $2^{3n}\doteq2^{3n+n^{1/4}}$, $2^{2^n}\not\doteq 2^{2^{n+\log n}}$. Note that $f(n)\asymp g(n)$ implies $f(n)\doteq g(n)$, but the converse is not true.

For any $q\in\bR_{>0}$, we write $\log_q(\cdot)$ for the logarithm to the base $q$. In particular, let $\log(\cdot)$ and $\ln(\cdot)$ denote logarithms to the base $2$ and $e$, respectively.

For any $\cA\subseteq\Omega$, the indicator function of $\cA$ is defined as, for any   $x\in\Omega$, 
\[\one_{\cA}(x)\coloneqq\begin{cases}1,&x\in \cA\\0,&x\notin \cA\end{cases}.\]
At times, we will slightly abuse notation by saying that $\one_{\sfA}$ is $1$ when event $\sfA$ happens and 0 otherwise. Note that $\one_{\cA}(\cdot)=\indicator{\cdot\in\cA}$.

\noindent\textbf{Sets.}
For any two nonempty sets $\cA$ and $\cB$ with addition and multiplication by a real scalar, let $\cA+\cB$ 
denote the Minkowski sum 
of them which is defined as
$\cA+\cB\coloneqq\curbrkt{a+b\colon a\in\cA,b\in\cB}$.
If $\cA=\{x\}$ is a singleton set, we write $x+\cB$ and
 for $\{x\}+\cB$.
For any $ r\in\bR $, the $r$-dilation of $\cA$ is defined as $ r\cA \coloneqq \curbrkt{r\va:\va\in\cA} $. 
In particular, $ -\cA\coloneqq (-1)\cA $. 

For $M\in\bZ_{>0}$, we let $[M]$ denote the set of first $M$ positive integers $\{1,2,\cdots,M\}$.

\noindent\textbf{Geometry.}
Let $ \normtwo{\cdot} $ denote the Euclidean/$\ell_2$-norm. Specifically, for any $\vx\in\bR^n$,
\[ \normtwo{\vx} \coloneqq\paren{\sum_{i=1}^n\vx(i)^2}^{1/2}.\]

With slight abuse of notation, we let $|\cdot|$ denote the ``volume'' of a set w.r.t.\ a measure that is obvious from the context. 
If $ \cA $ is a finite set, then $ |\cA| $ denotes the cardinality of $\cA$ w.r.t.\ the counting measure. 
For a set $ \cA\subset\bR^n $, let
\begin{align}
\aff(\cA) &\coloneqq \curbrkt{\sum_{i = 1}^k\lambda_i\va_i:k\in\bZ_{\ge1};\;\forall i\in[k], \va_i\in\cA,\lambda_i\in\bR,\sum_{i = 1}^k\lambda_i = 1} \notag 
\end{align}
denote the \emph{affine hull} of $\cA$, i.e., the smallest affine subspace containing $\cA$. 
If $ \cA $ is a connected compact set in $ \bR^n $ with nonempty interior and $ \aff(\cA) = \bR^n $, then $ |\cA| $ denotes the volume of $\cA$ w.r.t.\ the $n$-dimensional Lebesgue measure. 
If $ \aff(\cA) $ is a $k$-dimensional affine subspace for $ 1\le k<n $, then $ |\cA| $ denotes the $k$-dimensional Lebesgue volume of $\cA$. 

The closed $n$-dimensional Euclidean unit ball is defined as
\[\cB^n \coloneqq \curbrkt{\vy\in\bR^n\colon \normtwo{\vy} \le 1}.\]
The $(n-1)$-dimensional Euclidean unit sphere is defined as
\[\cS^{n-1} \coloneqq \curbrkt{\vy\in\bR^n\colon \normtwo{\vy} = 1}.\]
For any $ \vx\in\bR^n $ and $ r\in\bR_{>0} $, let $ \cB^n(r) \coloneqq r\cB^n, \cS^{n-1}(r) \coloneqq r\cS^{n-1} $ and $ \cB^n(\vx,r) \coloneqq \vx + r\cB^n, \cS^{n-1}(\vx,r) \coloneqq \vx + r\cS^{n-1} $. 

Let $ V_n \coloneqq |\cB^n| $.

\section{Preliminaries}
\label{sec:prelim}

\begin{definition}[Gamma function]
\label{def:gamma-fn}
For any $ z\in\bC $ with $ \Re(z)>0 $, the \emph{Gamma function} $ \Gamma(z) $ is defined as 
\begin{align}
\Gamma(z) &\coloneqq \int_0^\infty v^{z-1}e^{-v}\diff v. \notag 
\end{align}
\end{definition}

\begin{fact}[Volume of the unit ball]
$ V_n = \frac{\pi^{n/2}}{\Gamma(n/2+1)} \stackrel{n\to\infty}{\asymp} \frac{1}{\sqrt{\pi n}}(2\pi e/n)^{n/2} $. 
\end{fact}

\begin{lemma}[Markov]
\label{lem:markov}
If $\bfx$ is a nonnegative random variable, then for any $a>0$, $ \prob{\bfx\ge a}\le \expt{\bfx}/a $. 
\end{lemma}

\begin{lemma}[Chebyshev]
\label{lem:cheb}
if $\bfx$ is a nonnegative random variable, then for any $ a>0 $ and $ k>0 $, $ \prob{\bfx\ge a}\le\expt{\bfx^k}/a^k $. 
\end{lemma}


\begin{theorem}[\cramer]
\label{thm:cramer-ldp}
Let $ \curbrkt{\bfx_i}_{i = 1}^n $ be a sequence of i.i.d.\ real-valued random variables. 
Let $ \bfs_n \coloneq\frac{1}{n}\sum_{i = 1}^n\bfx_i $. 
Then for any closed $ \cF\subset\bR $,
\begin{align}
\limsup_{n\to\infty} \frac{1}{n}\ln\prob{\bfs_n\in \cF} &\le -\inf_{x\in \cF} \sup_{\lambda\in\bR}\curbrkt{ \lambda x - \ln\expt{e^{\lambda \bfx_1}} }; \notag 
\end{align}
and for any open $ \cG\subset\bR $, 
\begin{align}
\liminf_{n\to\infty} \frac{1}{n}\ln\prob{\bfs_n\in \cG} &\ge -\inf_{x\in \cG} \sup_{\lambda\in\bR}\curbrkt{\lambda x - \ln\expt{e^{\lambda \bfx_1}}}. \notag 
\end{align}
Furthermore, when $\cF$ or $\cG$ corresponds to the upper (resp.\ lower) tail of $ \bfs_n $, the maximizer $ \lambda\ge0 $ (resp.\ $ \lambda\le0 $).
\end{theorem}

\begin{theorem}[Stirling's approximation]
\label{thm:stirling}
The asymptotics of $ \Gamma(z+1) $ for large positive $z$ is given by $ \Gamma(z+1)\asymp\sqrt{2\pi z}(z/e)^z $. 
\end{theorem}

\begin{lemma}[Gaussian integral]
\label{lem:gauss-int}
Let $ A\in\bR^{n\times n} $ be a positive-definite matrix. Then 
\begin{align}
\int_{\bR^n} \exp\paren{-\vx^\top A\vx} \diff\vx &= \sqrt{\frac{\pi^n}{\det(A)}}. \notag 
\end{align}
\end{lemma}

\begin{lemma}[Matrix determinant lemma]
\label{lem:sherman-morrison}
Let $ A\in\bR^{n\times n} $ be a non-singular matrix and let $ \vu,\vv\in\bR^n $. 
Then
\begin{align}
\det\paren{A + \vu\vv^\top} &= \paren{1 + \vv^\top A^{-1}\vu}\det(A). \notag 
\end{align}
\end{lemma}

\section{Basic definitions and facts}
\label{sec:def}

Given the intimate connection between packing and error-correcting codes, we will interchangeably use the terms ``multiple packing'' and ``list-decodable code''. 
The parameter $ L\in\bZ_{\ge2} $ is called the \emph{multiplicity of overlap} or the \emph{list-size}. 
The parameters $N$ and $P$ (in the case of bounded packing) are called the \emph{input and noise power constraints}, respectively. 
Elements of a packing are called either \emph{points} or \emph{codewords}. 
We will call a size-$L$ subset of a packing an \emph{$L$-list}. 
This paper is only concerned with the fundamental limits of multiple packing for asymptotically large dimension $n\to\infty$. 
When we say ``a'' code $ \cC $, we will implicitly refer to an infinite sequence of codes $ \curbrkt{\cC_i}_{i\ge1} $ where $ \cC_i\subset\bR^{n_i} $ and $ \curbrkt{n_i}_{i\ge1} $ is an increasing sequence of positive integers. 
We call $\cC$ a \emph{spherical code} if $ \cC\subset\cS^{n-1}(\sqrt{nP}) $ and we call it a \emph{ball code} if $ \cC\subset\cB^n(\sqrt{nP}) $. 

In the rest of this section, we list a sequence of formal definitions and some facts associated with these definitions. 

\begin{definition}[Bounded multiple packing]
\label{def:packing-ball}
Let $ N,P>0 $ and $ L\in\bZ_{\ge2} $. 
A subset $ \cC \subseteq \cB^n( \sqrt{nP}) $ is called a \emph{$ (P,N,L-1) $-list-decodable code} (a.k.a.\ a \emph{$(P,N,L-1)$-multiple packing}) if for every $ \vy\in\bR^n $,
\begin{align}
\card{\cC\cap\cB^n(\vy, \sqrt{nN})} \le L-1 .
\label{eqn:packing-ball}
\end{align}
The \emph{rate} (a.k.a.\ \emph{density}) of $ \cC $ is defined as
\begin{align}
R(\cC) \coloneqq \frac{1}{n}{\ln\cardC} .
\label{eqn:density-bounded}
\end{align}
\end{definition}

\begin{definition}[Unbounded multiple packing]
\label{def:packing-euclidean}
Let $ N>0 $ and $ L\in\bZ_{\ge2} $. 
A subset $ \cC \subseteq \bR^n $ is called a \emph{$ (N,L-1) $-list-decodable code} (a.k.a.\ an \emph{$(N,L-1)$-multiple packing}) if for every $ \vy\in\bR^n $,
\begin{align}
\card{\cC\cap\cB^n(\vy, \sqrt{nN})} \le L-1 .
\label{eqn:packing-ball-repeat}
\end{align}
The \emph{rate} (a.k.a.\ \emph{density}) of $ \cC $ is defined as 
\begin{align}
R(\cC) \coloneqq& \limsup_{K\to\infty} \frac{1}{n}\ln\frac{\card{\cC\cap (K\cA)}}{\card{K\cA}}, 
\label{eqn:density-unbounded}
\end{align}
where $ \cA $ is an arbitrary centrally symmetric connected compact set in $ \bR^n $ with nonempty interior. 
\end{definition}

\begin{remark}
Common choices of $ \cA $ include the unit ball $ \cB^n $, the unit cube $ [-1,1]^n $, the fundamental Voronoi region $ \cV_\Lambda $ of a (full-rank) lattice $ \Lambda\subset\bR^n $, etc. 
Some choices of $\cA$ may be more convenient than the others for analyzing certain ensembles of packings. 
Therefore, we do not fix the choice of $\cA$ in \Cref{def:packing-euclidean}. 
\end{remark}

\begin{remark}
It is a slight abuse of notation to write $ R(\cC) $ to refer to the rate of either a bounded packing or an unbounded packing. 
However, the meaning of $ R(\cC) $ will be clear from the context. 
The rate of an unbounded packing (as per \Cref{eqn:density-unbounded}) is also called the \emph{normalized logarithmic density} in the literature. 
It measures the rate (w.r.t.\ \Cref{eqn:density-bounded}) per unit volume. 
\end{remark}

Note that the condition given by \Cref{eqn:packing-ball,eqn:packing-ball-repeat} is equivalent to that for any $ (\vx_1,\cdots,\vx_L)\in\binom{\cC}{L} $, 
\begin{align}
\bigcap_{i = 1}^L\cB^n(\vx_i,\sqrt{nN}) = \emptyset. \label{eqn:packing-ball-alternative} 
\end{align}

\begin{definition}[Chebyshev radius and average radius of a list]
\label{def:cheb-rad-avg-rad}
Let $ \vx_1,\cdots,\vx_L $ be $L$ points in $ \bR^n $. 
Then the \emph{squared Chebyshev radius} $ \rad^2(\vx_1,\cdots,\vx_L) $ of $ \vx_1,\cdots,\vx_L $ is defined as the (squared) radius of the smallest ball containing $ \vx_1,\cdots,\vx_L $, i.e., 
\begin{align}
\rad^2(\vx_1,\cdots,\vx_L) \coloneqq& \min_{\vy\in\bR^n} \max_{i\in[L]} \normtwo{\vx_i - \vy}^2. \label{eqn:cheb-rad} 
\end{align}
The \emph{average squared radius} $ \ol\rad^2(\vx_1,\cdots,\vx_L) $ of $ \vx_1,\cdots,\vx_L $ is defined as the average squared distance to the centroid, i.e., 
\begin{align}
\ol\rad^2(\vx_1,\cdots,\vx_L) \coloneqq& \frac{1}{L}\sum_{i = 1}^L{\normtwo{\vx_{i} - \vx}^2}, \label{eqn:avg-rad}
\end{align}
where $ \vx \coloneqq \frac{1}{L}\sum_{i = 1}^L\vx_i $ is the centroid of $ \vx_1,\cdots,\vx_L $. 
We refer to the square root of the average squared radius as the average radius of the list.
\end{definition}


\begin{remark}
One should note that for an $L$-list $\cL$ of points, the smallest ball containing $\cL$ is not necessarily the same as the \emph{circumscribed ball}, i.e., the ball such that all points in $\cL$ live on the boundary of the ball. 
The circumscribed ball of the polytope $ \conv\curbrkt{\cL} $ spanned by the points in $\cL$ may not exist. 
If it does exist, it is not necessarily the smallest one containing $ \cL $. 
However, whenever it exists, the smallest ball containing $\cL$ must be the circumscribed ball of a certain \emph{subset} of $\cL$. 
\end{remark}

\begin{remark}
\label{rk:motivation-avg-rad}
We remark that the motivation behind the definition of average squared radius (\Cref{eqn:avg-rad}) is to replace the maximization in \Cref{eqn:cheb-rad} with average. 
\begin{align}
\min_{\vy\in\bR^n} \exptover{\bfi\sim[L]}{\normtwo{\vx_\bfi - \vy}^2} 
&= \min_{\vy\in\bR^n} \frac{1}{L}\sum_{i = 1}^L \sum_{j = 1}^n \paren{\vx_i(j) - \vy(j)}^2 \notag\\
&= \min_{(y_1,\cdots,y_n)\in\bR^n} \sum_{j = 1}^n \frac{1}{L}\sum_{i = 1}^L\paren{\vx_i(j) - y_j}^2 \label{eqn:before-interchange} \\
&= \sum_{j = 1}^n \min_{y_j\in\bR} \frac{1}{L}\sum_{i = 1}^L\paren{\vx_i(j) - y_j}^2 \label{eqn:interchange} \\
&= \frac{1}{L}\sum_{i = 1}^L\sum_{j = 1}^n\paren{\vx_i(j) - y_j^*}^2 \label{eqn:opt-y} \\
&= \frac{1}{L} \sum_{i = 1}^L \normtwo{\vx_i - \vx}^2. \label{eqn:relax-formula}
\end{align}
\Cref{eqn:interchange} holds since the inner summation $ \frac{1}{L}\sum_{i = 1}^L \paren{\vx_i(j) - \vy(j)}^2 $ in \Cref{eqn:before-interchange} only depends on $ y_j $ among all $ y_1,\cdots,y_n $. 
\Cref{eqn:opt-y} follows since for each $j$, the minimizer of the minimization in \Cref{eqn:interchange} is $ y_j^* \coloneqq \frac{1}{L}\sum_{i = 1}^L\vx_i(j) $. 
In \Cref{eqn:relax-formula}, the minimizer $ \vy^* $ equals $ \vx \coloneqq \frac{1}{L}\sum_{i = 1}^L\vx_i $. 
\end{remark}

\begin{definition}[Chebyshev radius and average squared radius of a code]
\label{def:cheb-rad-avg-rad-code}
Given a code $ \cC\subset\bR^n $ of rate $R$, the \emph{squared $(L-1)$-list-decoding radius} of $\cC$ is defined as
\begin{align}
\rad^2_L(\cC) \coloneqq&  \min_{\cL\in\binom{\cC}{L}} \rad^2(\cL). \label{eqn:rad-code}
\end{align}
The \emph{$(L-1)$-average squared radius} of $\cC$ is defined as 
\begin{align}
\ol\rad^2_L(\cC) \coloneqq& \min_{\cL\in\binom{\cC}{L}} \ol\rad^2(\cL). \label{eqn:avgrad-code}
\end{align}
\end{definition}

\begin{definition}[Bounded average-radius multiple packing]
\label{def:avg-rad-multi-pack-bounded}
A subset $ \cC\subset\cB^n(\sqrt{nP}) $ is called a \emph{$ (P,N,L-1) $-average-radius list-decodable code} (a.k.a.\ a \emph{$ (P,N,L-1) $-average-radius multiple packing}) if $ \ol\rad^2_L(\cC)>nN $. 
The \emph{rate} (a.k.a.\ \emph{density}) $ R(\cC) $ of $ \cC $ is given by \Cref{eqn:density-bounded}. 
The \emph{$(P,N,L-1)$-average-radius list-decoding capacity} (a.k.a.\ \emph{$(P,N,L-1)$ average-radius multiple packing density}) is defined as
$$ \ol C_{L-1}(P,N) \coloneqq \limsup_{n\to\infty} \limsup_{\cC \subseteq \cB^n( \sqrt{nP})\colon \ol\rad^2_L(\cC)>nN} R(\cC) .$$ 
The \emph{squared $(L-1)$-average-radius list-decoding radius} at rate $R$ with input constraint $P$ is defined as
\begin{align}
\ol\rad^2_L(P,R) \coloneqq& \limsup_{n\to\infty} \limsup_{\cC \subseteq \cB^n( \sqrt{nP}) \colon R(\cC)\ge R} \ol\rad^2_L(\cC). \notag 
\end{align}
\end{definition}

\begin{definition}[Unbounded average-radius multiple packing]
\label{def:avg-rad-multi-pack-unbounded}
A subset $ \cC\subset\bR^n $ is called an \emph{$ (N,L-1) $-average-radius list-decodable code} (a.k.a.\ an \emph{$ (N,L-1) $-average-radius multiple packing}) if $ \ol\rad^2_L(\cC)>nN $. 
The \emph{rate} (a.k.a.\ \emph{density}) $ R(\cC) $ of $ \cC $ is given by \Cref{eqn:density-unbounded}. 
The \emph{$(N,L-1)$-average-radius list-decoding capacity} (a.k.a.\ \emph{$(N,L-1)$ average-radius multiple packing density}) is defined as
$$ \ol C_{L-1}(N) \coloneqq \limsup_{n\to\infty} \limsup_{\cC \subseteq \bR^n\colon \ol\rad^2_L(\cC)>nN} R(\cC) .$$ 
The \emph{squared $(L-1)$-average-radius list-decoding radius} at rate $R$ (without input constraint) is defined as
\begin{align}
\ol\rad^2_L(R) \coloneqq& \limsup_{n\to\infty} \limsup_{\cC \subseteq \bR^n \colon R(\cC)\ge R} \ol\rad^2_L(\cC). \notag 
\end{align}
\end{definition}

Note that $ (L-1) $-list-decodability defined by \Cref{eqn:packing-ball} or \Cref{eqn:packing-ball-alternative} is equivalent to $ \rad^2_L(\cC) > nN$. 
We also define the \emph{$(P,N,L-1)$-list-decoding capacity} (a.k.a.\ \emph{$(P,N,L-1)$-multiple packing density})
$$ C_{L-1}(P,N) \coloneqq \limsup_{n\to\infty} \limsup_{\cC \subseteq \cB^n( \sqrt{nP})\colon \rad^2_L(\cC)>nN} R(\cC) ,$$ 
and the \emph{squared $(L-1)$-list-decoding radius} at rate $R$ with input constraint $P$
\begin{align}
\rad^2_L(P,R) \coloneqq& \limsup_{n\to\infty} \limsup_{\cC \subseteq \cB^n( \sqrt{nP}) \colon R(\cC)\ge R} \rad^2_L(\cC), \notag 
\end{align}
and their unbounded analogues \emph{$(N,L-1)$-list-decoding capacity} (a.k.a.\ \emph{$(N,L-1)$-multiple packing density}) $ C_{L-1}(N) $ and the \emph{squared $(L-1)$-list-decoding radius} $ \rad^2_L(R) $ at rate $R$: 
\begin{align}
C_{L-1}(N) &\coloneqq \limsup_{n\to\infty} \limsup_{\cC \subseteq \bR^n\colon \rad^2_L(\cC)>nN} R(\cC) , \notag \\
\rad^2_L(R) &\coloneqq \limsup_{n\to\infty} \limsup_{\cC \subseteq \bR^n \colon R(\cC)\ge R} \rad^2_L(\cC). \notag 
\end{align}

Since the average radius is at most the Chebyshev radius, average-radius list-decodability is stronger than regular list-decodability. 
Any lower (resp.\ upper) bound on $ \ol C_{L-1}(P,N) $ (resp.\ $ C_{L-1}(P,N) $) is automatically a lower (resp.\ upper) bound on $ C_{L-1}(P,N) $ (resp.\ $ \ol C_{L-1}(P,N) $). 
The same relation also holds for the unbounded versions $ C_{L-1}(N) $ and $ \ol C_{L-1}(N) $. 
Proving upper/lower bounds on $ C_{L-1}(P,N) $ (resp.\ $ \ol C_{L-1}(P,N) $) is equivalent to proving upper/lower bounds on $ \rad^2_L(P,R) $ (resp.\ $ \ol\rad^2_L(P,R) $).
The same relation also holds for the unbounded versions $ C_{L-1}(N) $ (resp.\ $ \ol C_{L-1}(N) $) and $ \rad^2_L(R) $ (resp.\ $ \ol\rad^2_L(R) $).

\subsection{Density of bounded and unbounded packings}
\label{eqn:comments-density-bdd-unbdd}
In general (not only for multiple packing), it is customary to measure the density of a bounded constellation using \Cref{eqn:density-bounded} while to measure that of an unbounded constellation using \Cref{eqn:density-unbounded}. 
Informally, we expect them to be related to each other in the following way. 
Let $\cC_P$ be a constellation in $ \cB^n(\sqrt{nP}) $. 
Let $ \cC\coloneqq\lim\limits_{P\to\infty}\cC_P $. 
Then
\begin{align}
R(\cC) &= \limsup_{P\to\infty} \frac{1}{n}\ln\frac{\card{\cC\cap\cB^n( \sqrt{nP})}}{\card{\cB^n(\sqrt{nP})}} \notag \\
&= \limsup_{P\to\infty} \frac{1}{n}\ln\frac{\card{\cC_P}}{\card{\cB^n(\sqrt{nP})}} \notag \\
&= \limsup_{P\to\infty} \paren{R(\cC_P) - \frac{1}{n}\ln\card{\cB^n(\sqrt{nP})}} \notag \\
&\asymp \limsup_{P\to\infty} \paren{R(\cC_P) - \frac{1}{n}\ln\paren{\sqrt{nP}^n\cdot\frac{1}{\sqrt{\pi n}}\sqrt{\frac{2\pi e}{n}}^n}} \notag \\
&\asymp \limsup_{P\to\infty} \paren{R(\cC_P) - \frac{1}{2}\ln(2\pi eP)}. \label{eqn:relation-bounded-unbounded}
\end{align}
Suppose that for a certain problem without input constraint, the optimal density $ R^*(\cC^*) $ is achieved by an unbounded constellation $ \cC^* $, while for the same problem with input constraint $P$, the corresponding optimal density $ R^*(\cC_P^*) $ is achieved by a bounded constellation $ \cC_P^* $. 
The above relation between $ R^*(\cC^*) $ and $ R^*(\cC_P^*) $ given by \Cref{eqn:relation-bounded-unbounded} is empirically correct for some problems. 
However, the reasoning is not exactly rigorous since the bounded constellation $ \cC_P^* $ may not converge to an unbounded constellation as $ P $ approaches infinity. 

Let us look at some examples to justify the above comments. 
For the multiple packing problem with asymptotically large $L$, the capacity without input constraint is known to be $ \frac{1}{2}\ln\frac{1}{2\pi eN} $ \cite{zhang-vatedka-2019-listdecreal} whereas the capacity with input constraint $P$ equals $ \frac{1}{2}\ln\frac{P}{N} $ \cite{zhang-quadratic-arxiv}. 
These two quantities indeed satisfy the relation given by \Cref{eqn:relation-bounded-unbounded}. 

For the channel coding problem over an AWGN channel, the capacity without input constraint is $ \frac{1}{2}\ln\frac{1}{2\pi e N} $ whereas the capacity with input constraint $P$ equals $ \frac{1}{2}\ln\paren{1 + \frac{P}{N}} $. 
It still holds that 
\begin{align}
\lim_{P\to\infty} \paren{\frac{1}{2}\ln\paren{1+\frac{P}{N}} - \frac{1}{2}\ln(2\pi eP)} = \lim_{P\to\infty}\frac{1}{2}\ln\paren{\frac{1}{2\pi eN} + \frac{1}{2\pi eP}} = \frac{1}{2}\ln\frac{1}{2\pi eN}. \notag
\end{align}
However, from this example, we see that one cannot rearrange both sides of \Cref{eqn:relation-bounded-unbounded} and claim that for any $P$, 
\begin{align}
R(\cC_P) \asymp R(\cC)+ \frac{1}{2}\ln(2\pi eP). \notag
\end{align}

\subsection{More representations of the average radius}
\label{sec:comments-radius}

In this section, we present several different representations of the average radius. 
Some of them will be crucially used in the subsequent sections of this paper. 
These representations are summarized in the following theorem. 

\begin{theorem}
\label{thm:repr-rad}



Let $ L\in\bZ_{\ge2} $ and $ \vx_1,\cdots,\vx_L\in\bR^n $. 
Then the average squared radius of $ \vx_1,\cdots,\vx_L $ admits the following alternative representations:
\begin{enumerate}
\item $\begin{aligned}
\ol\rad^2(\vx_1,\cdots,\vx_L) &= \frac{1}{L}\sum_{i = 1}^L\normtwo{\vx_i}^2 - \normtwo{\vx}^2; \notag
\end{aligned}$

\item 
$\begin{aligned}
\ol\rad^2(\vx_1,\cdots,\vx_L) 
&= \frac{L-1}{L^2}\sum_{i = 1}^L\normtwo{\vx_i}^2 - \frac{1}{L^2}\sum_{(i,j)\in[L]^2:i\ne j} \inprod{\vx_i}{\vx_j}; \notag 
\end{aligned}$

\item 
$\begin{aligned}
\ol\rad^2(\vx_1,\cdots,\vx_L)
&= \frac{1}{2L^2}\sum_{(i,j)\in[L]^2:i\ne j} \normtwo{\vx_i - \vx_j}^2. \notag 
\end{aligned}$
\end{enumerate}
\end{theorem}

The proof follows from elementary manipulations of the definition of $ \ol\rad^2(\vx_1,\cdots,\vx_L) $. 
\begin{align}
\ol\rad^2(\vx_1,\cdots,\vx_L) &= \frac{1}{L}\sum_{i = 1}^L\normtwo{\vx_i - \vx}^2 \notag \\
&= \frac{1}{L}\sum_{i = 1}^L\inprod{\vx_i - \vx}{\vx_i - \vx} \notag \\
&= \frac{1}{L}\sum_{i = 1}^L\paren{ \normtwo{\vx_i}^2 - 2\inprod{\vx_i}{\vx} + \normtwo{\vx}^2 } \notag \\
&= \frac{1}{L}\sum_{i = 1}^L\normtwo{\vx_i}^2 - 2\inprod{\vx}{\vx} + \normtwo{\vx}^2 \notag \\
&= \frac{1}{L}\sum_{i = 1}^L\normtwo{\vx_i}^2 - \normtwo{\vx}^2. \label{eqn:avg-rad-alternative} 
\end{align}

The above expression can be further written as 
\begin{align}
\ol\rad^2(\vx_1,\cdots,\vx_L) &= \frac{1}{L}\sum_{i = 1}^L\normtwo{\vx_i}^2 - \normtwo{\vx}^2 \notag \\
&= \frac{1}{L}\sum_{i = 1}^L\normtwo{\vx_i}^2 - \frac{1}{L^2}\sum_{(i,j)\in[L]^2} \inprod{\vx_i}{\vx_j} \notag \\
&= \frac{1}{L}\sum_{i = 1}^L\normtwo{\vx_i}^2 - \frac{1}{L^2}\sum_{i = 1}^L\normtwo{\vx_i}^2 - \frac{1}{L^2}\sum_{(i,j)\in[L]^2:i\ne j} \inprod{\vx_i}{\vx_j} \notag \\
&= \frac{L-1}{L^2}\sum_{i = 1}^L\normtwo{\vx_i}^2 - \frac{1}{L^2}\sum_{(i,j)\in[L]^2:i\ne j} \inprod{\vx_i}{\vx_j}. \label{eqn:avg-rad-formula-pw-corr} 
\end{align}

At last, \Cref{eqn:avg-rad-formula-pw-corr} can in turn be rewritten as 
\begin{align}
\ol\rad^2(\vx_1,\cdots,\vx_L) 
&= \frac{L-1}{L^2}\sum_{i = 1}^L\normtwo{\vx_i}^2 - \frac{1}{L^2}\sum_{(i,j)\in[L]^2:i\ne j} \inprod{\vx_i}{\vx_j} \notag \\
&= \frac{L-1}{L^2}\sum_{i = 1}^L\normtwo{\vx_i}^2 - \frac{1}{2L^2} \sum_{(i,j)\in[L]^2:i\ne j}\paren{\normtwo{\vx_i}^2 + \normtwo{\vx_j}^2} + \frac{1}{2L^2} \sum_{(i,j)\in[L]^2:i\ne j} \paren{\normtwo{\vx_i}^2 + \normtwo{\vx_j}^2 - 2\inprod{\vx_i}{\vx_j}} \notag \\
&= \frac{L-1}{L^2}\sum_{i = 1}^L\normtwo{\vx_i}^2 - \frac{1}{L^2} \sum_{(i,j)\in[L]^2:i\ne j}\normtwo{\vx_i}^2 + \frac{1}{2L^2}\sum_{(i,j)\in[L]^2:i\ne j} \normtwo{\vx_i - \vx_j}^2 \notag \\
&= \frac{1}{2L^2}\sum_{(i,j)\in[L]^2:i\ne j} \normtwo{\vx_i - \vx_j}^2 . \label{eqn:avg-rad-pw-dist} 
\end{align}

If all $ \vx_1,\cdots,\vx_L $ have the same $ \ell_2 $ norm $ \sqrt{nP} $, then \Cref{eqn:avg-rad-alternative}
\begin{align}
\ol\rad^2(\vx_1,\cdots,\vx_L) = nP - \normtwo{\vx}^2. \label{eqn:avg-rad-spherical-formula}
\end{align}
and \Cref{eqn:avg-rad-formula-pw-corr} becomes
\begin{align}
\ol\rad^2(\vx_1,\cdots,\vx_L) = \frac{L-1}{L}nP - \frac{1}{L^2}\sum_{(i,j)\in[L]^2:i\ne j} \inprod{\vx_i}{\vx_j} . \label{eqn:avg-rad-spherical-formula-pw-corr} 
\end{align}

\subsection{Reduction from ball codes to spherical codes}
\label{sec:reduction-ball-to-spherical}
We show that $ C_{L-1}(P,N) $ can be achieved using spherical codes. 
\begin{theorem}
\label{thm:reduction-ball-to-spherical}
Let $ N,P>0 $ and $ L\in\bZ_{\ge2} $. 
If there is a $ (P,N,L-1) $-list-decodable code $ \cC\subset\cB^n(\sqrt{nP}) $ of rate $ R $, then there must exist a $ \paren{P,\frac{nN}{N+1},L-1} $-list-decodable code $ \cC\subset\cS^n(\vec 0,\sqrt{(n+1)P}) $\footnote{We use $ \vec0\in\bR^{n+1} $ to denote the all-zero vector of length $ n+1 $.} of rate $ R $. 
\end{theorem}

\begin{proof}
Let $ \cC\subset\cB^n(\sqrt{nP}) $ be a $ (P,N,L-1) $-list-decodable \emph{ball} code whose codewords may not necessarily have the same norm. 
We will turn $ \cC $ into a $ (P,N',L-1) $-list-decodable \emph{spherical} code $ \cC'\subset\cS^n(\vec 0,\sqrt{(n+1)P}) $ for some $ N'\xrightarrow{n\to\infty}N $. 
We will append one more coordinate to each codeword in $ \cC $ in such a way that the augmented codeword of length $n+1$ has norm exactly $ \sqrt{(n+1)P} $. 
The collection of augmented codewords obtained from $ \cC $ is denoted by $ \cC' $. 
Specifically, for a codeword $ \vx\in\cC $, let $ \vec x\in\cC' $ be defined as
\begin{align}
\vec x(i) &= \begin{cases}
\vx(i), & i\in[n] \\
\sqrt{(n+1)P - \normtwo{\vx}^2}, & i = n+1
\end{cases}. \notag 
\end{align}
for $ i\in[n+1] $. 
It is immediate that $ \normtwo{\vec x} = \sqrt{(n+1)P} $. 
By list-decodability of $ \cC $, we know that for any $ (\vx_1,\cdots,\vx_L)\in\binom{\cC}{L} $, 
\begin{align}
\rad^2(\vx_1,\cdots,\vx_L) = \min_{\vy\in\bR^n} \max_{i\in[L]} \normtwo{\vx_i - \vy}^2 > nN. \notag 
\end{align}
We would like to check the list-decodability of $ \cC' $. 
Let $ (\vec x_1,\cdots,\vec x_L) $ denote the augmented version of $ (\vx_1,\cdots,\vx_L) $. 
\begin{align}
\rad^2(\vec x_1,\cdots,\vec x_L)
&= \min_{\vec y\in\bR^{n+1}} \max_{i\in[L]} \normtwo{\vec x_i - \vec y}^2 \notag \\
&= \min_{\substack{\vy\in\bR^n \\ y_{n+1}\in\bR}} \max_{i\in[L]} \curbrkt{ \normtwo{\vx_i - \vy}^2 + \abs{\vec x(n+1) - y_{n+1}}^2  } \notag \\
&\ge \min_{\vy\in\bR^n} \max_{i\in[L]} \normtwo{\vx_i - \vy}^2 \notag \\
&>nN \notag \\
&= (n+1)N', \notag 
\end{align}
where $ N' \coloneqq \frac{nN}{n+1} \xrightarrow{n\to\infty} N $. 
Therefore, $ \cC'\subset\cS^n(\vec 0,\sqrt{(n+1)P}) $ is $ (P,N',L-1) $-list-decodable. 
\end{proof}

The same trick applies to the unbounded case as well. 
Using the representation of average squared radius given in \Cref{eqn:avg-rad-pw-dist}, 
we have
\begin{align}
\ol\rad^2(\vec x_1,\cdots,\vec x_L) 
&= \frac{1}{2L^2}\sum_{(i,j)\in[L]^2:i\ne j} \normtwo{\vec x_i - \vec x_j}^2 \notag \\
&= \frac{1}{2L^2}\sum_{(i,j)\in[L]^2:i\ne j} \paren{\normtwo{\vx_i - \vx_j}^2 + \abs{\vec x_i(n+1) - \vec x_j(n+1)}^2} \notag \\
&\ge \frac{1}{2L^2}\sum_{(i,j)\in[L]^2:i\ne j} \normtwo{\vx_i - \vx_j}^2 \notag \\
&= \rad^2(\vx_1,\cdots,\vx_L). \notag 
\end{align}
We conclude that $ \ol C_{L-1}(P,N) $ can also be achieved using spherical codes.

\begin{remark}
\label{rk:cohn-zhao-reduction}
For $L = 2$, Cohn and Zhao \cite{cohnzhao-2014} obtained a reduction from unbounded packings to bounded packings for \emph{finite} $n$. 
They first restricted an unbounded packing to a ball code in $ \cB^n(\sqrt{nP}) $ and then proved a reduction from ball codes to spherical codes. 
They showed that for \emph{any} $ n\ge1 $, and any $ P,N>0 $ such that $ 1/4\le N/P\le1 $, 
\begin{align}
\limsup_{\cC\subset\bR^n:\rad^2_2(\cC)>nN} \Delta(\cC) &\le \sqrt{N/P}^n \limsup_{\cC\subset\cS^{n-1}(\sqrt{nP}):\rad^2_2(\cC)>nN} R(\cC), \notag 
\end{align}
where the density of an unbounded packing is measured using 
the fraction of space occupied by the balls of radius $ \sqrt{nN} $. 
The above inequality holds for \emph{every} dimension with additional constraints on $ N/P $. 
On the other hands, our reduction (\Cref{thm:reduction-ball-to-spherical}) only holds for \emph{asymptotically large} dimensions without any constraint on $ N/P $. 
We are unable to generalize the result for finite $n$ by Cohn and Zhao to $L>3$. 
\end{remark}

\section{Random coding with expurgation}
\label{sec:rand-cod-exp}
In this section, we introduce a basic strategy for proving lower bounds on the list-decoding capacity. 
The idea, known as \emph{random coding with expurgation}, is classical in information theory and coding theory. 

\begin{theorem}
\label{thm:rand-cod-expurg}
Let $ P,N>0 $ and $ L\in\bZ_{\ge2} $. 
For any distribution $ \cD $ supported on $ \cB^n(\sqrt{nP}) $, 
\begin{align}
C_{L-1}(P,N) &\ge \frac{E(P,N,L)}{L-1}, \quad 
\ol C_{L-1}(P,N) \ge \frac{\ol E(P,N,L)}{L-1}, \notag 
\end{align}
where $ E(P,N,L) $ and $ \ol E(P,N,L) $ depending on $\cD$ are lower bounds on the exponents of the following probabilities
\begin{align}
E(P,N,L) &\ge \lim_{n\to\infty} -\frac{1}{n}\ln\probover{(\vbfx_1,\cdots,\vbfx_L)\sim\cD^\tl}{\rad^2(\vbfx_1,\cdots,\vbfx_L)\le nN}, \notag \\
\ol E(P,N,L) &\ge \lim_{n\to\infty} -\frac{1}{n}\ln\probover{(\vbfx_1,\cdots,\vbfx_L)\sim\cD^\tl}{\ol\rad^2(\vbfx_1,\cdots,\vbfx_L)\le nN}. \notag 
\end{align}
\end{theorem}

\begin{proof}
Let us first consider constructions of $ (P,N,L-1) $-multiple packings.
Let $ \cD $ be a distribution on $ \cB^n(\sqrt{nP}) $. 
We sample $ M = e^{nR} $ points i.i.d.\ from $\cD$ and call the point set $ \cC $. 
Fix any $ (i_1,\cdots,i_L)\in\binom{[M]}{L} $. 
Suppose we manage to compute the following probability 
\begin{align}
\prob{\phi(\vbfx_{i_1},\cdots,\vbfx_{i_L})\le nN} \doteq e^{-nE(P,N,L)}, \label{eqn:tail-of-rad-to-comp} 
\end{align}
where $ \phi(\cdot) $ denotes either the squared Chebyshev radius $ \rad^2(\cdot) $ or the average squared radius $ \ol\rad^2(\cdot) $. 
Then we know that the expected number of \emph{bad} lists (i.e., $L$-lists $ \cL $ for which $ \phi(\cL) $ is at most $ nN $) is 
\begin{align}
\expt{\card{\curbrkt{\cL\in\binom{\cC}{L}:\phi(\cL)\le nN}}} &\doteq \binom{M}{L}e^{-nE(P,N,L)} 
\le M^Le^{-nE(P,N,L)} 
= e^{n(RL - E(P,N,L))}. \notag 
\end{align}
We then expurgate one point from each of these bad lists and hope that the total number of removed points is at most, say, $ M/2 $. 
This is satisfied if 
\begin{align}
&& e^{n(RL - E(P,N,L) + o(1))} &\le M/2 \notag \\
\impliedby && RL - E(P,N,L) &\le R - o(1) \notag \\
\impliedby && R &\le \frac{E(P,N,L)}{L-1} - o(1). \notag 
\end{align}
Note that the expurgation process destroys all bad $L$-lists, and the remaining subcode $ \cC'\subset\cC $ of size at least $ M/2 $ is indeed a $ (P,N,L-1) $-multiple packing. 
Therefore, there exists a construction of $ (P,N,L-1) $-multiple packing with rate approaching $ \frac{E(P,N,L)}{L-1} $. 
\end{proof}

Technically, if the distribution $ \cD $ is not exactly supported in $ \cB^n(\sqrt{nP}) $, e.g., $ \cD = \cN(\vzero,PI_n) $, then we also need to {show that the probability}
\begin{align}
\probover{\vbfx\sim\cD}{\normtwo{\vbfx}>\sqrt{nP}} &\doteq e^{-nF(P,N)} \notag
\end{align}
for some exponent $ F(P,N) $
and argue that the expected number of points with norm larger than $ \sqrt{nP} $ is negligible, i.e., 
\begin{align}
\expt{\card{\cC\cap\cB^n(\sqrt{nP})^c}} &\doteq e^{n(R - F(P,N))} \ll e^{nR} . \notag 
\end{align}
Such a probability is usually very small and the number of points with large norm that will be expurgated is therefore negligible. 
As a result, this expurgation step does not put effective constraints on the rate $R$. 

The above idea can be adapted to $ (N,L-1) $-multiple packings. 

\begin{remark}
From \Cref{eqn:tail-of-rad-to-comp}, if we do not proceed with expurgation but instead take a union bound over $ (i_1,\cdots,i_L)\in\binom{[M]}{L} $, then the bound we eventually get will be $ \frac{E(P,N,L)}{L} $ which is much less than $ \frac{E(P,N,L)}{L-1} $ especially for small $L$. 
In fact, it can be shown \cite{gmrsw-2020-sharp-threshold} that $ \frac{E(P,N,L)}{L} $ is the \emph{threshold rate} of purely random codes in the sense that below the threshold such codes are list-decodable with high probability and above the threshold such codes are list-decodable with vanishing probability. 
Therefore, expurgation is necessary for random coding arguments. 
\end{remark}

\section{Lower bounds via Gaussian codes}
\label{sec:lb-gaussian}

In this section, we analyze average-radius list-decodability of a Gaussian code (with expurgation)
and establish the following theorem.
\begin{theorem}
\label{thm:lb-gaussian}
For any $ P,N\in\bR_{>0} $ such that $ \frac{N}{P} \le\frac{L-1}{L} $ and $ L\in\bZ_{\ge2} $, the $ (P,N,L-1) $-average-radius list-decoding capacity is at least 
\begin{align}
\ol C_{L-1}(P,N) &\ge \frac{1}{2}\paren{\ln\frac{(L-1)P}{LN} + \frac{LN}{(L-1)P} - 1}. \label{eqn:bound-gaussian}
\end{align}
\end{theorem}

\begin{remark}
The above bound (\Cref{eqn:bound-gaussian}) converges to $ \frac{1}{2}\paren{\ln\frac{P}{N} + \frac{N}{P} - 1} $ as $ L\to\infty $. 
The limiting value is strictly smaller than $ \frac{1}{2}\ln\frac{P}{N} $ which is the list-decoding capacity for asymptotically large $L$ (see \Cref{sec:listdec-cap-large}). 
When $ L=2 $, the above bound becomes $ \frac{1}{2}\paren{\ln\frac{P}{2N} + \frac{2N}{P} - 1} $ which is lower than the best known $ \frac{1}{2}\ln\frac{P^2}{4N(P-N)} $ (see \Cref{thm:gv}). 
Furthermore, the above bound shoots to infinity at $ N/P = 0 $ and hits zero at $ N/P = \frac{L-1}{L} $. 
These two values are tight. 
Specifically, $ \ol C_{L-1}(P,N) = \infty $ at $ N/P=0 $, obviously.
Also, it turns out that $ \ol C_{L-1}(P,N) = 0 $ at the point $ N/P = \frac{L-1}{L} $ known as the \emph{Plotkin point} (see \Cref{sec:ub}). 
\end{remark}

\begin{proof}
Construct a codebook by sampling $ M = e^{nR} $ (for some rate $ R>0 $ to be determined later) Gaussian vectors $ \vbfx_1,\cdots,\vbfx_M $, each entry i.i.d.\ from $ \cN\paren{0,\frac{P}{1+\eps}} $ where $ \eps $ is a small number to be determined. 
Let $ P'\coloneqq \frac{P}{1+\eps} $. 
The probability that a codeword violates the power constraint is 
\begin{align}
\probover{\vbfx\sim\cN(\vzero,P'I_n)}{\normtwo{\vbfx} > \sqrt{nP}}. \label{eqn:power-constr-bound}
\end{align}

To obtain the exponent of the probability in \Cref{eqn:power-constr-bound}, we first recall the Moment Generating Function (MGF) of a chi-square random variable.

\begin{definition}
\label{def:chi-sqaure}
The \emph{chi-square distribution} $ \chi^2(k) $ with \emph{degree of freedom} $k$ is defined as the distribution of $ \sum_{i = 1}^k\bfg_i^2 $ where $ \bfg_i\iid\cN(0,1) $ for $ 1\le i\le k $. 
\end{definition}

\begin{fact}
\label{fact:mgf-chi-sq}
If $ \bfx\sim\chi^2(k) $, then for $ \lambda<1/2 $
\begin{align}
\expt{e^{\lambda\bfx}} &= \sqrt{1-2\lambda}^{-k}. \notag 
\end{align}
\end{fact}

Plugging the formula in \Cref{fact:mgf-chi-sq} into \cramer's theorem (\Cref{thm:cramer-ldp}), we get the first order asymptotics of the tail of a chi-square random variable. 
\begin{lemma}
\label{fact:asymp-chi-sq}
If $ \bfx\sim\chi^2(k) $, then 
\begin{align}
\lim_{k\to\infty}\frac{1}{k}\ln\prob{\bfx>(1+\delta)k} &= \frac{1}{2}(-\delta+\ln(1+\delta)), && \text{for $ \delta>0 $}; \notag \\
\lim_{k\to\infty}\frac{1}{k}\ln\prob{\bfx<(1-\delta)k} &= \frac{1}{2}(\delta+\ln(1-\delta)), && \text{for $ \delta\in(0,1) $}. \notag 
\end{align}
\end{lemma}
Then the probability in \Cref{eqn:power-constr-bound} is 
\begin{align}
\probover{\vbfx\sim\cN(\vzero,P'I_n)}{\normtwo{\vbfx/\sqrt{P'}}^2 > nP/P'} &= \probover{\vbfx\sim\cN(\vzero,P'I_n)}{\normtwo{\vbfx/\sqrt{P'}}^2 > n(1+\eps)} \doteq e^{\frac{n}{2}\paren{-\eps + \ln\paren{1+\eps}}}, \notag 
\end{align}
since $ \normtwo{\vbfx/\sqrt{P'}}^2\sim\chi^2(n) $. 
Therefore the expected number of codewords that violate the power constraint is 
\begin{align}
\expt{\card{\curbrkt{i\in[M]: \normtwo{\vbfx_i}>\sqrt{nP}}}} &\doteq Me^{\frac{n}{2}\paren{-\eps + \ln\paren{1+\eps}}} = e^{n\paren{R + \frac{1}{2}\paren{-\eps + \ln\paren{1+\eps}}}}. \notag 
\end{align}
If $ \eps $ is set to be sufficiently small, the above number is negligible compared to the size $ e^{nR} $ of the code. 

We now compute the probability that the average squared radius of a list is too small (\Cref{eqn:tail-of-rad-to-comp} with $ \phi(\cdot) = \ol\rad^2(\cdot) $). 
\begin{align}
\prob{\ol\rad^2(\vbfx_1,\cdots,\vbfx_L)\le nN}
&= \prob{\frac{1}{L}\sum_{i = 1}^L\normtwo{\vbfx_i}^2 - \normtwo{\vbfx}^2\le nN} \label{eqn:apply-avg-rad-alternative} \\
&= \prob{\sum_{i = 1}^L\normtwo{\vbfx_i}^2 - L\normtwo{\vbfx}^2 \le LnN} \notag \\
&= \prob{\sum_{j = 1}^n\sum_{i = 1}^L\vbfx_i(j)^2 - L\sum_{j = 1}^n\paren{\frac{1}{L}\sum_{i = 1}^L\vbfx_i(j)}^2 \le LnN} \notag \\
&= \prob{ \sum_{j = 1}^n\paren{\sum_{i = 1}^L\vbfx_i(j)^2 - \frac{1}{L}\paren{\sum_{i = 1}^L\vbfx_i(j)}^2} \le LnN }, \label{eqn:to-use-quadratic-form}
\end{align}
where \Cref{eqn:apply-avg-rad-alternative} is by \Cref{eqn:avg-rad-alternative}
and $ \vbfx \coloneqq \frac{1}{L}\sum_{i = 1}^L\vbfx_i $. 

Define 
\begin{align}
g(x_1, \cdots,x_L) \coloneqq& \sum_{i = 1}^Lx_i^2 - \frac{1}{L}\paren{\sum_{i = 1}^Lx_i}^2. \notag 
\end{align}
The above probability (\Cref{eqn:to-use-quadratic-form}) can be rewritten as 
\begin{align}
\prob{\sum_{j = 1}^ng(\vbfx_1(j), \cdots,\vbfx_L(j)) \le LnN} \label{eqn:avg-rad-bd-g}
\end{align}
where $ \vbfx_i(j)\iid\cN(0,P') $ for each $i\in[L]$ and $j\in[n]$.

We note that the function $ g(\vec t) $ is a quadratic form of $ \vec t\in\bR^L $. 
Indeed, 
\begin{align}
g(\vec t) &= \sum_{i = 1}^L\vec t(i)^2 - \frac{1}{L}\paren{\sum_{i = 1}^L\vec t(i)}^2 
= \paren{1 - \frac{1}{L}} \sum_{i = 1}^L\vec t(i)^2 - \frac{2}{L} \sum_{i,j\in[L]: i<j} \vec t(i)\vec t(j) 
= \vec t^\top A\vec t, \label{eqn:def-g}
\end{align}
where 
\begin{align}
A \coloneqq& I_L - \frac{1}{L}J_L \in\bR^{L\times L} \notag
\end{align}
and $ I_L $ denotes the $ L\times L $ identity matrix and $ J_L $ denotes the $ L\times L $ all-one matrix. 

Let $ \vec\bfx_j \coloneqq (\vbfx_1(j), \cdots, \vbfx_L(j))\in\bR^L $. 
By the above manipulation, we have $ g(\vec\bfx_j) = \vec\bfx_j^\top A\vec\bfx_j $, where $ \vec\bfx_j \sim\cN(\vec 0, P'I_L) $. 
In fact, since $A$ is idempotent (i.e., $ A^2 = A $), such a quadratic form is distributed according to $ g(\vec\bfx_j)/P'\sim \chi^2(\tr(A)) = \chi^2(L-1) $. 
Furthermore, the sum of i.i.d.\ chi-square random variables is still chi-square distributed and the degrees of freedom are summed. 
Therefore, $ \sum_{j = 1}^n g(\vec\bfx_j)/P'\sim\chi^2( (L-1)n) $.

Assume $ \frac{N}{P}<\frac{L-1}{L} $. 
Let $ N = \frac{L-1}{L}P(1-\rho) $ for some $ \rho>0 $. 
One can solve $\rho$ and get $ \rho = 1 - \frac{LN}{(L-1)P} $. 
We now can apply \Cref{fact:asymp-chi-sq} to obtain the exact asymptotics of \Cref{eqn:avg-rad-bd-g}:
\begin{align}
\Cref{eqn:avg-rad-bd-g}
&= \prob{\sum_{j=1}^ng(\vec\bfx_j)/P' \le LnN/P'} \notag \\
&= \prob{\chi^2((L-1)n) \le Ln\cdot\frac{L-1}{L}P(1-\rho)\cdot\frac{1+\eps}{P}} \notag \\
&= \prob{\chi^2( (L-1)n) \le n(L-1)(1-\rho)(1+\eps)} \notag \\
&= \prob{\chi^2( (L-1)n) \le n(L-1)(1-\delta)} \label{eqn:def-delta} \\
&\doteq \exp\paren{(L-1)n\frac{1}{2}(\ln(1-\delta) + \delta)}. \notag 
\end{align}
In \Cref{eqn:def-delta}, we define $ \delta>0 $ such that $ 1-\delta = (1 - \rho)(1+\eps) $, i.e., 
$\delta \coloneqq (1+\eps)\rho - \eps$. 
Note that $ \delta\xrightarrow{\eps\to0}\rho $. 

The expected number of \emph{bad} lists is therefore
\begin{align}
\expt{\card{\curbrkt{\cL\in\binom{[M]}{L}: \ol\rad^2\paren{\curbrkt{\vbfx_i}_{i\in\cL}}\le nN}}} &\doteq \binom{M}{L}e^{n\frac{L-1}{2}\paren{\ln(1-\delta) + \delta}} \le e^{n\paren{LR + \frac{L-1}{2}\paren{\ln(1-\delta) + \delta}}}. \notag 
\end{align}

To obtain an $(P,N,L-1)$-average-radius list-decodable code, we first throw away all codewords that violate the power constraint.
We then throw away one codeword in each \emph{bad} list whose average squared radius is at most $ nN $. 
We would like the expected number of codewords we threw away in total not to exceed $ M/2 $. 
This can be satisfied if the rate $R$ satisfies
\begin{align}
LR + \frac{L-1}{2}\paren{\ln(1-\delta) + \delta} \le R - o(1). \notag 
\end{align}
That is, (ignoring the $o(1)$ factor)
\begin{align}
R &< \frac{1}{2}\paren{\ln\frac{1}{1-\delta} - \delta} \notag \\
&= \frac{1}{2}\paren{\ln\frac{1}{(1-\rho)(1+\eps)} - (1+\eps)\rho + \eps} \notag \\
&= \frac{1}{2}\paren{\ln\frac{(L-1)P}{LN(1+\eps)} - (1+\eps)\paren{1 - \frac{LN}{(L-1)P}} + \eps} \notag \\
&\xrightarrow{\eps\to0} \frac{1}{2}\paren{\ln\frac{(L-1)P}{LN} + \frac{LN}{(L-1)P} - 1}. \notag 
\end{align}
We then obtain a $(P,N,L-1)$-average-radius list-decodable code of size at least $ M/2 $. 
\end{proof}

\section{Lower bounds via spherical codes}
\label{sec:lb-spherical}
In this section, we examine the $ (P,N,L-1) $-average-radius list-decodability of uniformly random spherical codes (with expurgation). 
We first sample $ M = e^{nR} $ codewords $ \vbfx_1, \cdots, \vbfx_M $ uniformly and independently from $ \cS^{n-1}( \sqrt{nP}) $. 
We are again interested in the asymptotics of \Cref{eqn:tail-of-rad-to-comp} with $ \phi(\cdot) = \ol\rad^2(\cdot) $. 
In \Cref{sec:spherical-normal-apx,sec:spherical-khinchin} respectively, we will obtain a bound on the normalized exponent of \Cref{eqn:tail-of-rad-to-comp} via two different approaches: \emph{normal approximation} and \emph{Khinchin's inequality} (with sharp constants).
Plugging this bound into the random coding with expurgation framework, we get the following result. 
\begin{theorem}
\label{thm:lb-spherical}
For any $ P,N\in\bR_{>0} $ such that $ \frac{N}{P} \le\frac{L-1}{L} $ and $ L\in\bZ_{\ge2} $, the $ (P,N,L-1) $-average-radius list-decoding capacity is at least 
\begin{align}
\ol C_{L-1}(P,N) &\ge \frac{1}{2}\paren{1 - \frac{LN}{(L-1)P} - \frac{1}{L-1}\ln\frac{L(P-N)}{P}}. \label{eqn:bound-spherical}
\end{align}
\end{theorem}

\begin{remark}
\label{rk:worse-spherical}
We make several observations on the behaviour of the bound in \Cref{thm:lb-spherical}. 
\begin{enumerate}
\item 
We observe that the bound \Cref{eqn:bound-spherical} for spherical codes is the second largest lower bound for sufficiently large rates. 
The largest lower bound turns out to be \Cref{eqn:compare-lb-ee,eqn:compare-lb-spherical-improved} (for all $N,P\ge0$ and $ L\in\bZ_{\ge2} $).
It can be proved using error exponents as a proxy \cite{zhang-split-ee}, or, as we shall see in \Cref{sec:spherical-improved}, using refined analysis for expurgated spherical codes per se. 
Indeed, the only difference between these two bounds \Cref{eqn:compare-lb-ee,eqn:bound-spherical} is that 
\begin{align}
1 - \frac{LN}{(L-1)P} &\le \ln\frac{(L-1)P}{LN}, \label{eqn:compare-spherical-vs-ee}
\end{align}
where the LHS appears in \Cref{eqn:bound-spherical} and the RHS appears in \Cref{eqn:compare-lb-ee}. 
The above inequality follows from the elementary inequality $ \ln x\le x - 1 $ for any $ x>0 $. 
In fact, the LHS forms the first two terms of the Taylor series of the RHS at $ N/P = (L-1)/L $. 

\item 
However, \Cref{eqn:bound-spherical} gets worse quickly as $ N/P $ decreases, as plotted in \Cref{fig:compare-low,fig:compare-mid}. 
In particular, for $ N/P = 0 $, $ \ol C_{L-1}(P,N) $ should be infinity since clearly one can pack arbitrarily many balls of radius zero. 
However, the above bound attains a finite value (specifically $ \frac{1}{2}\paren{1 - \frac{\ln L}{L-1}} $) at $ N/P = 0 $. 
This will be explained in \Cref{eqn:spherical-bad-low-rate}. 

\item 
The above bound (\Cref{eqn:bound-spherical}) converges to $ \frac{1}{2}\paren{1 - \frac{N}{P}} $ as $ L\to\infty $. 
The limiting value is strictly smaller than $ \frac{1}{2}\ln\frac{P}{N} $ which is the list-decoding capacity for asymptotically large $L$ (see \Cref{sec:listdec-cap-large}). 

\item 
On the other hand, when $ L=2 $, the above bound becomes $ \frac{1}{2}\paren{1-\frac{2N}{P} - \ln\frac{2(P-N)}{P}} $ which is lower than the best know bound $ \frac{1}{2}\ln\frac{P^2}{4N(P-N)} $ in \Cref{thm:gv}. 
\end{enumerate}
\end{remark}

\Cref{thm:lb-spherical} will be proved in \Cref{sec:spherical-normal-apx,sec:spherical-khinchin} using two different methods: normal approximation and Khinchin's inequality. 
However, the observations in \Cref{rk:worse-spherical} suggest that the analysis in the proof of \Cref{thm:lb-spherical} is likely suboptimal. 
We then prove in \Cref{sec:spherical-improved} the following theorem which shows that (expurgated) spherical codes can in fact achieve a much higher rate.
Indeed, this rate is the exact asymptotics of this code ensemble and matches the best known lower bound \Cref{eqn:compare-lb-ee} which itself was originally proved using a completely different method (cf.\ \cite{zhang-split-ee}). 
On the other hand, the advantage of the techniques in \Cref{sec:spherical-khinchin} is that it can be easily extended to related ensembles such as (expurgated) ball codes (cf.\ \Cref{sec:ball-codes}), which is not the case for the techniques to be presented in \Cref{sec:spherical-improved}. 

\begin{theorem}
\label{thm:lb-spherical-improved}
For any $ P,N\in\bR_{>0} $ such that $ \frac{N}{P} \le\frac{L-1}{L} $ and $ L\in\bZ_{\ge2} $, the $ (P,N,L-1) $-average-radius list-decoding capacity is at least 
\begin{align}
\ol C_{L-1}(P,N)  &\ge  \frac{1}{2}\sqrbrkt{\ln\frac{(L-1)P}{LN}  +  \frac{1}{L-1}\ln\frac{P}{L(P-N)}}. \label{eqn:bound-spherical-improved}
\end{align}
\end{theorem}

\begin{remark}
\label{rk:lb-asymp}
When $ L\to\infty $, the above bound (\Cref{eqn:bound-spherical}) converges to the list-decoding capacity $ \frac{1}{2}\ln\frac{P}{N} $ for $L\to\infty$ (see \Cref{sec:listdec-cap-large}). 
In fact, it can be shown that list-size $ L = \frac{1}{\eps}\ln\frac{1}{\eps} + 1 $ is sufficient to achieve $(P,N,L-1)$-list-decoding rate $ \frac{1}{2}\ln\frac{P}{N} - \cO(\eps) $ for any $ \eps\le\frac{1}{2} $ (which ensures $ \frac{1}{\eps}\ln\frac{1}{\eps} + 1\ge2 $). 
To see this, note that 
\begin{align}
\frac{1}{2}\sqrbrkt{\ln\frac{(L-1)P}{LN}  +  \frac{1}{L-1}\ln\frac{P}{L(P-N)}}
&= \frac{1}{2}\ln\frac{P}{N} - \frac{1}{2}\paren{\frac{\ln L}{L-1} - \frac{c}{L-1} + \ln\frac{L}{L-1}} , \notag 
\end{align}
where we let $ c \coloneqq \ln\frac{P}{P-N}>0 $. 
It then suffices to upper bound the term in parentheses on the RHS assuming $ L = \frac{1}{\eps}\ln\frac{1}{\eps} + 1 \ge 2 $. 
\begin{align}
\frac{\ln L}{L-1} - \frac{c}{L-1} + \ln\frac{L}{L-1} 
&= \frac{\ln L}{L-1} - \frac{c}{L-1} + \ln\paren{1 + \frac{1}{L-1}} \notag \\
&\le \frac{\ln L}{L-1} - \frac{c}{L-1} + \frac{1}{L-1} \label{eqn:bound-log-1} \\
&\le \frac{\ln\paren{\frac{1}{\eps}\ln\frac{1}{\eps} + 1}}{\frac{1}{\eps}\ln\frac{1}{\eps}} - \frac{c-1}{\frac{1}{\eps}\ln\frac{1}{\eps}} \notag \\
&= \frac{\ln\paren{\frac{1}{\eps}\ln\frac{1}{\eps}} + \ln\paren{1 + \frac{1}{\frac{1}{\eps}\ln\frac{1}{\eps}}}}{\frac{1}{\eps}\ln\frac{1}{\eps}} - \frac{c-1}{\frac{1}{\eps}\ln\frac{1}{\eps}} \notag \\
&\le \frac{\ln\frac{1}{\eps} + \ln\ln\frac{1}{\eps} + \frac{1}{\frac{1}{\eps}\ln\frac{1}{\eps}}}{\frac{1}{\eps}\ln\frac{1}{\eps}} - \frac{c-1}{\frac{1}{\eps}\ln\frac{1}{\eps}} \label{eqn:bound-log} \\
&= \eps + \eps\cdot\frac{\ln\ln\frac{1}{\eps}}{\ln\frac{1}{\eps}} + \paren{\frac{\eps}{\ln\frac{1}{\eps}}}^2 - \frac{(c-1)\eps}{\ln\frac{1}{\eps}} \notag \\
&= \cO(\eps) , \notag 
\end{align}
provided that $\eps$ is independent of $c$. 
\Cref{eqn:bound-log,eqn:bound-log-1} both follow from the elementary inequality $ \ln(1+x)\le x $ for any $x>-1$. 
This is to be contrasted with the necessary condition $ L = \Omega(\eps) $ to be derived in \Cref{rk:ub-asymp} from \Cref{thm:eb}. 

For $ L=2 $, \Cref{eqn:bound-spherical} recovers the best known bound $ \frac{1}{2}\ln\frac{P^2}{4N(P-N)} $ in \cite{blachman-1962}. 
Furthermore, it is tight at $ N/P=0 $ where the optimal density is $\infty$ and $ N/P=\frac{L-1}{L} $ where the optimal density is $ 0 $. 
\end{remark}

\subsection{Proof {of \Cref{thm:lb-spherical}} via normal approximation}
\label{sec:spherical-normal-apx}
We need the following lemma for approximating the density of a sum of uniform spherical vectors by that of a Gaussian (with a suitably chosen variance). 
\begin{lemma}[\cite{molavianjazi-laneman-2015-second-order-gmac}, Proposition 2]
\label{lem:apx-spherical}
Let $ P>0 $. 
Let $ \vbfx_1, \cdots, \vbfx_L $ be i.i.d.\ random vectors uniformly distributed on $ \cB^n( \sqrt{nP}) $. 
Let $ \wt\vbfx_1, \cdots, \wt\vbfx_L $ be i.i.d.\ Gaussian vectors each with distribution $ \cN(\vzero, P\cdot I_n) $. 
Let $ \vbfy \coloneqq \sum_{i = 1}^L\vbfx_i $ and $ \wt\vbfy \coloneqq\sum_{i = 1}^L\wt\vbfx_i $ respectively. 
Then for any $ \vy\in\bR^n $, we have
$ {P_{\vbfy}(\vy)}/{P_{\wt\vbfy}(\vy)} \le \kappa $, 
where $ \kappa = \kappa(P)>0 $ is a constant depending only on $P$ and $L$ but not $n$. 
\end{lemma}

Assume $ N = \frac{L-1}{L}P(1-\rho) $ for some constant $ \rho>0 $.
Equivalently, $ \rho = 1 - \frac{LN}{(L-1)P} $. 
We are ready to compute \Cref{eqn:tail-of-rad-to-comp} with $ \phi(\cdot) = \ol\rad^2(\cdot) $. 
\begin{align}
\prob{\ol\rad^2(\vbfx_1,\cdots,\vbfx_L)\le nN} &= \prob{ nP - \normtwo{\vbfx}^2 \le nN } \label{eqn:define-x} \\
&= \prob{ \normtwo{\vbfx}^2 \ge n(P - N) } \notag \\
&= \prob{ \normtwo{\vbfx}^2 \ge n\paren{\frac{P}{L} + \paren{1 - \frac{1}{L}}P\rho} } \notag \\
&= \prob{ \normtwo{\vbfy}^2 \ge nLP(1 + \eta) } \label{eqn:define-y} \\
&= \int_{\bR^n} P_{\vbfy}(\vy)\indicator{\normtwo{\vy}^2 \ge nLP(1 + \eta)} \diff\vy  \notag \\
&\le \kappa \int_{\bR^n} P_{\wt\vbfy}(\vy)\indicator{\normtwo{\vy}^2 \ge nLP(1 + \eta)} \diff\vy \label[ineq]{ineq:apx-spherical} \\
&= \kappa\prob{\normtwo{\wt\vbfy}^2 \ge nLP(1 + \eta)}, \label{eqn:define-tilde-y} 
\end{align}
In \Cref{eqn:define-x}, 
$ \vbfx \coloneqq \frac{1}{L}\sum_{i = 1}^L\vbfx_i $. 
In \Cref{eqn:define-y}, we define $\vbfy \coloneqq L\vbfx = \sum_{i = 1}^L\vbfx_i$ 
and $ \eta\coloneqq (L-1)\rho $.  
\Cref{ineq:apx-spherical} is by \Cref{lem:apx-spherical} and $ \wt\vbfy = \sum_{i = 1}^L \wt\vbfx_i $ where $ \vbfx_i\iid\cN(\vzero, P\cdot I_n) $ for $ 1\le i\le L $. 
Note that the probability in \Cref{eqn:define-tilde-y} is nothing but the tail of a chi-square random variable. 
Indeed, 
\begin{align}
\prob{\normtwo{\wt\vbfy}^2 \ge nLP(1 + \eta)} &= \prob{ \sum_{j = 1}^n \wt\vbfy(j)^2 \ge nLP(1 + \eta) } \notag \\
&= \prob{\sum_{j = 1}^n \paren{\sum_{i = 1}^L\wt\vbfx_i(j)}^2 \ge nLP(1 + \eta)} \notag \\
&= \prob{\sum_{j = 1}^n \cN_j(0, LP)^2 \ge nLP(1 + \eta)} \notag \\
&= \prob{\sum_{j = 1}^n \cN_j(0, 1)^2 \ge n(1 + \eta)} \notag \\
&= \prob{\chi^2(n) \ge n(1 + \eta)} \notag \\
&\doteq \exp\paren{\frac{n}{2}\paren{-\eta + \ln(1 + \eta)}}. \label{eqn:bound-normal-apx} 
\end{align}
The last equality (\Cref{eqn:bound-normal-apx}) is by \Cref{fact:asymp-chi-sq}. 
Going through the same expurgation process as described in \Cref{sec:rand-cod-exp}, we obtain a $(P,N,L-1)$-average-radius list-decodable code of rate asymptotically $R$ as long as 
\begin{align}
LR + {\frac{1}{2}\paren{-\eta + \ln(1 + \eta)}} &\le R, \notag 
\end{align}
Since $ \rho = 1 - \frac{LN}{(L-1)P} $, this is equivalent to
\begin{align}
R &\le \frac{1}{2(L-1)}\paren{\eta - \ln(1 + \eta)} \notag \\
&= \frac{1}{2}\paren{\rho - \frac{1}{L - 1}\ln(1 + (L-1)\rho)} \notag \\
&= \frac{1}{2}\paren{1 - \frac{LN}{(L-1)P} - \frac{1}{L-1}\ln\paren{1 + (L-1) - \frac{LN}{P}}} \notag \\
&= \frac{1}{2}\paren{1 - \frac{LN}{(L-1)P} - \frac{1}{L-1}\ln\frac{L(P-N)}{P}}. \notag 
\end{align}
This finishes the proof of \Cref{thm:lb-spherical}.

\begin{remark}
\label{eqn:spherical-bad-low-rate}
Note that \Cref{eqn:bound-spherical} equals $ \frac{1}{2}\paren{1 - \frac{\ln L}{L-1}}<\infty $ when $ N/P = 0 $. 
However, if $ N/P = 0 $, \Cref{eqn:define-x} is obviously zero since $ \normtwo{\vbfx}^2\le\paren{\frac{1}{L}\cdot L\sqrt{nP}}^2 = nP $ for any $ \vbfx_1, \cdots,\vbfx_L\in\cB^n(\sqrt{nP}) $. 
This means that we should be able to achieve an arbitrarily large $R$ in this case. 
The reason why \Cref{eqn:bound-spherical} is incorrect at high rates is that \Cref{ineq:apx-spherical} is loose when $ N/P $ is small. 
In particular, if $ N/P = 0 $, we have $ \rho = 1 $, therefore $ \eta = L-1 $.
The probability $ \prob{\normtwo{\wt\vbfx}^2\ge nP} $, where $ \wt\vbfx = \frac{1}{L}\sum_{i = 1}^L\wt\vbfx_i $ and each $ \wt\vbfx_i $ is i.i.d.\ Gaussian, equals $ \prob{\chi^2(n)\ge nL} $. 
This probability has a finite (rather than infinite) exponent. 

Given the above observation, it is tempting to consider better versions of normal approximation in \Cref{ineq:apx-spherical}. 
The simplest refinement would be to truncate the p.d.f.\ at $ \normtwo{\vy} = L\sqrt{nP} $, i.e., 
\begin{align}
P_{\vbfy}(\vy) &\le \kappa P_{\wt\vbfy}(\vy) \indicator{\normtwo{\vy}^2\le nL^2P}. \notag 
\end{align}
Unfortunately, the above refinement does not give rise to an improved bound on the achievable rate. 
\end{remark}

\subsection{Proof {of \Cref{thm:lb-spherical}} via Khinchin's inequality}
\label{sec:spherical-khinchin}
It was proved in \cite[Theorem 3]{konig-2001-khinchin-spherical} that the $p$-th moment of the norm of the sum of a collection of spherical vectors is \emph{upper bounded}\footnote{Note that the $p$-th moment is trivially \emph{lower bounded} by the second moment simply by the monotonicity (in $p$) of moment $ \paren{\expt{(\cdot)^p}}^{1/p} $.} by it second moment. 
\begin{theorem}[\cite{konig-2001-khinchin-spherical}]
\label{thm:khinchin-spherical}
Let $ 1\le p<\infty $ and $ n\ge2 $. 
If $ \vbfx_1,\cdots,\vbfx_k $ are independent and uniformly distributed on $ \cS^{n-1} $, then for any $ k\in\bZ_{\ge1} $ and $ \vec\alpha = [\vec\alpha(1),\cdots,\vec\alpha(k)]\in\bR^k $,
\begin{align}
\paren{\expt{\normtwo{\sum_{i = 1}^k\vec\alpha(i)\vbfx_i}^p}}^{1/p} &\le C_{n,p} \normtwo{\vec\alpha}, \label[ineq]{eqn:khinchin-ineq}
\end{align}
where 
\begin{align}
C_{n,p} \coloneqq& \sqrt{\frac{2}{n}}\paren{\frac{\Gamma\paren{\frac{p+n}{2}}}{\Gamma\paren{\frac{n}{2}}}}^{1/p}. \notag 
\end{align}
\end{theorem}

\begin{remark}
The constant $ C_{n,p} $ (independent of $ \vec\alpha $) is the best possible and is attained when $ \vec\alpha = [1/\sqrt{k},\cdots,1/\sqrt{k}]\in\bR^k $ and $ k\to\infty $. 
\end{remark}

\begin{remark}
It is not hard to check that $ \normtwo{\vec\alpha} = \sqrt{\expt{\normtwo{\sum_{i = 1}^k\vec\alpha(i)\vbfx_i}^2}} $.
Indeed, 
\begin{align}
\paren{\expt{\normtwo{\sum_{i = 1}^k\vec\alpha(i)\vbfx_i}^2}}^{1/2}
&= \paren{ \expt{\sum_{(i,j)\in[k]^2}\vec\alpha(i)\vec\alpha(j)\inprod{\vbfx_i}{\vbfx_j}} }^{1/2} \notag \\
&= \paren{\expt{\sum_{i\in[k]}\vec\alpha(i)^2\normtwo{\vbfx_i}^2} + \expt{\sum_{(i,j)\in[k]^2:i\ne j}\vec\alpha(i)\vec\alpha(j)\inprod{\vbfx_i}{\vbfx_j}}}^{1/2} \notag \\
&= \paren{\sum_{i\in[k]}\vec\alpha(i)^2 + \sum_{(i,j)\in[k]^2:i\ne j}\vec\alpha(i)\vec\alpha(j)\sum_{\ell\in[n]}\expt{\vbfx_i(\ell)\vbfx_j(\ell)}}^{1/2} \notag \\
&= \paren{\sum_{i\in[k]}\vec\alpha(i)^2 + \sum_{(i,j)\in[k]^2:i\ne j}\vec\alpha(i)\vec\alpha(j)\sum_{\ell\in[n]}\expt{\vbfx_i(\ell)}\expt{\vbfx_j(\ell)}}^{1/2} \label{eqn:2mmt-indep} \\
&= \paren{\sum_{i\in[k]}\vec\alpha(i)^2}^{1/2} \notag \\
&= \normtwo{\vec\alpha}. \notag 
\end{align}
\Cref{eqn:2mmt-indep} follows since $ \vbfx_i $ and $ \vbfx_j $ are independent for $ 1\le i\ne j\le k $. 
Therefore, \Cref{eqn:khinchin-ineq} is a moment comparison inequality. 
Also, the constant $ C_{n,p} $ in \Cref{eqn:khinchin-ineq} is in fact equal to $ C_{n,p} = \paren{\expt{\normtwo{\vbfg}^p}}^{1/p} $ where $ \vbfg\sim\cN(\vzero, I_n/n) $. 
\end{remark}

We can now get an upper bound on \Cref{eqn:define-x} by combining \Cref{thm:khinchin-spherical} and the Chebyshev's inequality (\Cref{lem:cheb}). 
Let $ \vbfy \coloneqq \sum_{i = 1}^L\vbfx_i $. 
Then
\begin{align}
\prob{\ol\rad^2(\vbfx_1,\cdots,\vbfx_L)\le nN} 
&= \prob{\normtwo{\vbfy}^2 \ge nLP(1+\eta)} \notag \\
&= \prob{\normtwo{\sum_{i = 1}^L\vbfu_i}\ge \sqrt{L(1+\eta)}} \label{eqn:define-u} \\
&= \prob{\normtwo{\sum_{i = 1}^L\vbfu_i}^p\ge(L(1+\eta))^{p/2}} \notag \\
&\le \frac{\expt{\normtwo{\sum_{i = 1}^L\vbfu_i}^p}}{(L(1+\eta))^{p/2}} . \label{eqn:to-apply-khinchin} 
\end{align}
In \Cref{eqn:define-u}, $ \vbfu_i\iid \unif(\cS^{n-1}) $ for $ 1\le i\le L $. 
To apply \Cref{thm:khinchin-spherical}, we let $ k = L $, $ \vec\alpha = [1,\cdots,1]\in\bR^L $. 
Note that $ \normtwo{\vec\alpha} = \sqrt{L} $. 
Therefore \Cref{eqn:to-apply-khinchin} is at most
\begin{align}
\Cref{eqn:to-apply-khinchin} &\le \frac{C_{n,p}^pL^{p/2}}{(L(1+\eta))^{p/2}} \notag \\
&= \paren{\frac{2}{n}}^{p/2} \frac{\Gamma\paren{\frac{p+n}{2}}}{\Gamma\paren{\frac{n}{2}}} (1+\eta)^{-p/2} \notag \\
&\stackrel{n\to\infty}{\asymp} \paren{\frac{2}{n}}^{p/2} 
\frac{\sqrt{\pi(p+n-2)} \paren{\frac{\frac{p+n}{2}-1}{e}}^{\frac{p+n}{2}-1}}{\sqrt{\pi(n-2)} \paren{\frac{\frac{n}{2}-1}{e}}^{\frac{n}{2}-1}} (1+\eta)^{-p/2} \label{eqn:stirling-gamma} \\
&= \paren{\frac{2}{n}}^{p/2} 
\sqrt{\frac{n+p-2}{n-2}} \paren{\frac{n+p-2}{2e}}^{\frac{n+p}{2} - 1}\paren{\frac{2e}{n-2}}^{\frac{n}{2}-1} (1+\eta)^{-p/2} \notag \\
&= \paren{\frac{n+p-2}{ne}}^{p/2} \sqrt{\frac{n+p-2}{n-2}}
\paren{\frac{n+p-2}{n-2}}^{\frac{n}{2}-1} (1+\eta)^{-p/2} \notag \\
&= \paren{\frac{n+p-2}{ne(1+\eta)}}^{p/2} \sqrt{\frac{n+p-2}{n-2}} \paren{1+\frac{p}{n-2}}^{\frac{n}{2}-1} . \label{eqn:to-set-p}
\end{align}
In \Cref{eqn:stirling-gamma}, we use the Stirling's approximation (\Cref{thm:stirling}) for Gamma functions. 
If we set $ p = \alpha n $ for some $ \alpha>0 $ and let $ A \coloneqq e(1+\eta) $, $ \beta = 1+\alpha $, then asymptotically in $ n\to\infty $, \Cref{eqn:to-set-p} becomes
\begin{align}
\Cref{eqn:to-set-p} &\stackrel{n\to\infty}{\asymp} \paren{\frac{1+\alpha}{e(1+\eta)}}^{n\alpha/2} \sqrt{1+\alpha} (1+\alpha)^{n/2-1} \notag \\
&= (1+\alpha)^{-1/2} \paren{\frac{1+\alpha}{e(1+\eta)}}^{n\alpha/2} (1+\alpha)^{n/2} \notag \\
&= (1+\alpha)^{-1/2} \exp\paren{\frac{n}{2}\paren{\alpha\ln(1+\alpha) - \alpha\ln A + \ln(1+\alpha)}} \notag \\
&= (1+\alpha)^{-1/2} \exp\paren{\frac{n}{2}\paren{(1+\alpha)\ln(1+\alpha) - (1+\alpha)\ln A+\ln A}} \notag \\
&= \beta^{-1/2}\exp\paren{\frac{n}{2}\paren{\beta(\ln\beta-\ln A) + \ln A}}. \notag
\end{align}
The exponent is maximized at $ \beta = e^{\ln A - 1} = A/e = 1+\eta $, that is, $ \alpha = \eta $. 
The corresponding maximum exponent (normalized by $1/n$) is 
\begin{align}
\frac{1}{2}\paren{(1+\eta)(\ln(1+\eta)-\ln(e(1+\eta))) + \ln(e(1+\eta))} = \frac{1}{2}(-(1+\eta) + \ln(1+\eta) + 1) = \frac{1}{2}(-\eta+\ln(1+\eta)) \notag
\end{align}
which matches \Cref{eqn:bound-normal-apx}. 
The rest of the proof remains the same and we get the same bound.

\subsection{Ball codes}
\label{sec:ball-codes}
Since the $p$-th moment of the norm of the sum of vectors uniformly distributed on unit \emph{sphere} is related to that of the sum of vectors uniformly distributed in unit \emph{ball}, we also have Khinchin's inequality for the latter quantity. 

\begin{theorem}[\cite{konig-2001-khinchin-spherical}, Proposition 4]
Let $ 0<p<\infty $, $ n\in\bZ_{\ge1} $ and $ r>0 $. 
Let $ \vec x_1,\cdots,\vec x_k\in\bR^{n+2} $ be independent random vectors uniformly distributed on $ \cS^{n+1} $. 
Let $ \vbfx_1,\cdots,\vbfx_k\in\bR^n $ be independent random vectors uniformly distributed in $ r\cB^n $. 
Then for any $ k\in\bZ_{\ge1} $ and $ (\alpha_1,\cdots,\alpha_k)\in\bR^k $, 
\begin{align}
\expt{\normtwo{\sum_{i = 1}^k\alpha_i\vbfx_i}^p} &= r^p\frac{n}{n+p} \expt{\normtwo{\sum_{i = 1}^k\alpha_i\vec x_i}^p}. \notag 
\end{align}
\end{theorem}

\begin{theorem}[\cite{konig-2001-khinchin-spherical}, Theorem 5]
\label{thm:khinchin-ball}
Let $ 1\le p<\infty $ and $ n\ge1 $. 
If $ \vbfx_1,\cdots,\vbfx_k $ are independent and uniformly distributed on $ \cB^n $, then for any $ k\in\bZ_{\ge1} $ and $ \vec\alpha = [\vec\alpha(1),\cdots,\vec\alpha(k)]\in\bR^k $,
\begin{align}
\paren{\expt{\normtwo{\sum_{i = 1}^k\vec\alpha(i)\vbfx_i}^p}}^{1/p} &\le \wt C_{n,p} \normtwo{\vec\alpha}, \notag
\end{align}
where 
\begin{align}
\wt C_{n,p} \coloneqq& \sqrt{\frac{2}{n+2}}\paren{\frac{\Gamma\paren{\frac{p+n}{2}}}{\Gamma\paren{\frac{n}{2}}}}^{1/p}. \notag 
\end{align}
\end{theorem}

Therefore, random \emph{ball} codes with expurgation can attain the same bound as \Cref{eqn:bound-spherical} which was attained by random \emph{spherical} codes with expurgation.

\subsection{Improved analysis of spherical codes}
\label{sec:spherical-improved}

In this section, we develop improved analysis of uniformly random spherical codes (with expurgation), showing that they in fact achieve the best known capacity lower bound (cf.\ \Cref{eqn:compare-lb-spherical-improved}) even under $ (P,N,L-1) $-\emph{average-radius} list-decoding.  
We are again interested in the asymptotics of \Cref{eqn:define-x}. 
We will obtain the exact normalized exponent of \Cref{eqn:define-x} via a combination of \cramer's large deviation theorem (\Cref{thm:cramer-ldp}) and Gallager's trick \cite{gallager-1965-simple-deriv} for bounding indicator functions. 
Plugging this bound into the random coding with expurgation framework, we get \Cref{thm:lb-spherical-improved}. 

Let $ \delta>0 $ be a sufficiently small constant. 
Define the following truncated Gaussian distribution $ Q\in\Delta(\bR^n) $: 
\begin{align}
Q(\vx) &\coloneqq Z_n^{-1} \phi_P^{\ot n}(\vx) \indicator{-\delta\le\sum_{i = 1}^n(\vx(i)^2 - P)\le0}, \notag 
\end{align}
where $ \phi_P(\cdot) $ is the (one-dimensional) Gaussian p.d.f.\ with variance $P$, 
and the normalizing constant $ Z_n $ is defined as
\begin{align}
Z_n &\coloneqq \int_{\bR^n} \phi_P^{\ot n}(\vx) \indicator{-\delta\le\sum_{i = 1}^n(\vx(i)^2 - P)\le0} \diff\vx. \label{eqn:def-zn} 
\end{align}
Note that as $\delta\to0$, $Q$ converges to $ \unif(\cS^{n-1}(\sqrt{nP})) $. 
Recall Gallager's trick for bounding the indicator function: $ \indicator{a>1}\le a $ for any $ a>0 $. 
Applying this to the $ x,t\in\bR $ and $ \lambda>0 $, we have 
\begin{align}
\indicator{x\ge t} &= \indicator{e^{\lambda x} \ge e^{\lambda t}} \le e^{\lambda(x-t)} . \notag 
\end{align}
This yields the following bound on the indicator in \Cref{eqn:def-zn}: for any $ s\ge0 $, 
\begin{align}
\indicator{-\delta\le\sum_{i = 1}^n(\vx(i)^2 - P)\le0} 
&\le \indicator{\sum_{i = 1}^n(\vx(i)^2 - P)\ge-\delta} 
\le \exp\paren{s\sum_{i = 1}^n(\vx(i)^2-P) + s\delta}. \label{eqn:gallager-trick} 
\end{align}
Similarly, we can bound the indicator of the error event $\indicator{\ol\rad^2(\vx_1,\cdots,\vx_L)\le nN}$ as follows: for any $\lambda\ge0$, 
\begin{align}
\indicator{\ol\rad^2(\vx_1,\cdots,\vx_L)\le nN}
= \indicator{\sum_{j = 1}^n \vec{x}_j^\top A\vec{x}_j \le LnN}
\le \exp\paren{-\lambda\sum_{j = 1}^n \vec{x}_j^\top A\vec{x}_j + \lambda LnN} . \label{eqn:gallager-trick-err-event} 
\end{align}
Here the equality follows since the average squared radius is a quadratic form (cf.\ \Cref{eqn:avg-rad-bd-g,eqn:def-g}), and $ \vec{x}_j\in\bR^L $ denotes the vector $ [\vx_1(j),\cdots,\vx_L(j)] $. 

We are interested in understanding the large deviation exponent of  
\begin{align}
\probover{(\vbfx_1,\cdots,\vbfx_L)\sim Q^\tl}{\ol\rad^2(\vbfx_1,\cdots,\vbfx_L)\le nN}. \label{eqn:spherical-tail} 
\end{align}
The above tail probability can be written explicitly as 
\begin{align}
& \int_{(\bR^n)^L} Z_n^{-L} \prod_{i = 1}^L \paren{\phi_P^{\ot n}(\vx_i) \indicator{-\delta\le\sum_{j = 1}^n(\vx_i(j)^2 - P)\le0}} 
 \indicator{\ol\rad^2(\vx_1,\cdots,\vx_L)\le nN} \diff(\vx_1,\cdots,\vx_L) \notag \\
&\le Z_n^{-L} \int_{(\bR^n)^L} \prod_{i = 1}^L \sqrbrkt{\phi_P^{\ot n}(\vx_i) \exp\paren{s\sum_{j = 1}^n(\vx_i(j)^2-P) + s\delta}} 
 \indicator{\ol\rad^2(\vx_1,\cdots,\vx_L)\le nN} 
 \diff(\vx_1,\cdots,\vx_L) \label{eqn:use-gallager-trick} \\
&= Z_n^{-L} e^{Ls\delta} \int_{\bR^{Ln}} \prod_{i = 1}^L \prod_{j = 1}^n \sqrbrkt{\phi_P(\vx_i(j)) \exp\paren{s(\vx_i(j)^2 - P)}} 
 \indicator{\ol\rad^2(\vx_1,\cdots,\vx_L)\le nN} \diff(\{\vx_i(j)\}_{i\in[L],j\in[n]}) \notag \\
&\le Z_n^{-L} e^{Ls\delta} \int_{(\bR^L)^n} \prod_{j = 1}^n\sqrbrkt{\phi_P^{\ot L}(\vec{x}_j)\exp\paren{s\sum_{i=1}^L(\vec{x}_j(i)^2-P)}} 
 \exp\paren{-\lambda\sum_{j = 1}^n \vec{x}_j^\top A\vec{x}_j + \lambda LnN}
 \diff(\vec{x}_1,\cdots,\vec{x}_n) \label{eqn:intro-col-vec} \\
&= Z_n^{-L} e^{Ls\delta} e^{\lambda LnN} \int_{(\bR^L)^n} \prod_{j = 1}^n \left[\phi_P^{\ot L}(\vec{x}_j) 
 \exp\paren{s\sum_{i=1}^L(\vec{x}_j(i)^2-P)} \exp\paren{-\lambda\vec{x}_j^\top A\vec{x}_j}\right] \diff(\vec{x}_1,\cdots,\vec{x}_n) \notag \\
&= Z_n^{-L} e^{Ls\delta} e^{\lambda LnN} \prod_{j = 1}^n \int_{\bR^L} \phi_P^{\ot L}(\vec{x}_j) 
 \exp\paren{s\sum_{i=1}^L(\vec{x}_j(i)^2-P) - \lambda\vec{x}_j^\top A\vec{x}_j} \diff\vec{x}_j \notag \\
&= Z_n^{-L} e^{Ls\delta} e^{\lambda LnN} e^{-sLPn} 
 \sqrbrkt{\int_{\bR^L} \phi_P^{\ot L}(\vec{x})\exp\paren{s\sum_{i=1}^L\vec{x}(i)^2 - \lambda\vec{x}^\top A\vec{x}} \diff\vec{x}}^n \notag \\
&= Z_n^{-L} e^{Ls\delta} e^{\lambda LnN} e^{-sLPn} \sqrt{2\pi P}^{-Ln} 
 \sqrbrkt{\int_{\bR^L} \exp\paren{\paren{-\frac{1}{2P} + s}\sum_{i=1}^L\vec{x}(i)^2 - \lambda\vec{x}^\top A\vec{x}} \diff\vec{x}}^n. \label{eqn:int-in-paren} 
\end{align}
\Cref{eqn:use-gallager-trick} is by \Cref{eqn:gallager-trick} and \Cref{eqn:intro-col-vec} is by \Cref{eqn:gallager-trick-err-event}.  

The integral in the brackets in \Cref{eqn:int-in-paren} is a Gaussian integral (\Cref{lem:gauss-int}) which can be evaluated as follows:
\begin{align}
& \int_{\bR^L} \exp\paren{\paren{-\frac{1}{2P} + s}\sum_{i=1}^L\vec{x}(i)^2 - \lambda\vec{x}^\top A\vec{x}} \diff\vec{x} \label{eqn:integral-sph-improved} \\
&= \int_{\bR^L} \exp\paren{\vec{x}^\top\paren{\paren{-\frac{1}{2P} + s}I_L - \lambda A}\vec{x}} \diff\vec{x} \notag \\
&= \int_{\bR^L} \exp\paren{- \vec{x}^\top\paren{\paren{\frac{1}{2P} - s + \lambda}I_L - \frac{\lambda}{L}J_L}\vec{x}} \diff\vec{x} \notag \\
&= \sqrt{\frac{\pi^L}{\det\paren{\paren{\frac{1}{2P} - s + \lambda}I_L - \frac{\lambda}{L}J_L}}} , \label{eqn:int-tbd} 
\end{align}
where $ I_L $ denotes the $ L\times L $ identity matrix and $ J_L $ denotes the $ L\times L $ all-one matrix. 
The determinant in the denominator can be computed using \Cref{lem:sherman-morrison}: 
\begin{align}
\det\paren{\paren{\frac{1}{2P} - s + \lambda}I_L - \frac{\lambda}{L}J_L} 
 &= \paren{\frac{1}{2P} - s + \lambda}^L \det\paren{I_L - \frac{\lambda}{L\paren{\frac{1}{2P} - s + \lambda}}\one_L\one_L^\top} \notag \\
 &= \paren{\frac{1}{2P} - s + \lambda}^L \paren{1 - \frac{\lambda}{L\paren{\frac{1}{2P} - s + \lambda}}\one_L^\top\one_L} \notag \\
 &= \paren{\frac{1}{2P} - s + \lambda}^{L-1} \paren{\frac{1}{2P} - s} , \notag  
\end{align}
where $ \one_L $ denotes the all-one vector of length $L$. 
Therefore, continuing with \Cref{eqn:int-tbd}, we have that the integral in \Cref{eqn:integral-sph-improved} equals
\begin{align}
\sqrt{\frac{\pi^L}{\paren{\frac{1}{2P} - s + \lambda}^{L-1} \paren{\frac{1}{2P} - s}}}. \label{eqn:int-result} 
\end{align}
Next we argue that for any constant $ \delta>0 $, the normalizing constant $ Z_n $ in fact scales polynomially as $n\to\infty$. 
\begin{lemma}
\label{lem:bound-z}
Let $ P,\sigma,\delta>0 $ be constants. 
Let $ \phi_P(x)\coloneqq\frac{e^{-x^2/2P}}{\sqrt{2\pi P}} $ be the Gaussian density with variance $P$. 
Let $ f(x) \coloneqq x^2 - P $. 
Let $ Z_n $ be defined by \Cref{eqn:def-zn}. 
Then
$Z_n \stackrel{n\to\infty}{\asymp} \frac{\delta}{2P\sqrt{\pi n}}$. 
\end{lemma}

\begin{proof}
The proof follows from the central limit theorem. 
\begin{align}
Z_n &= \int_{\bR^n} \phi_P^\tn(\vx) \indicator{-\delta\le\sum_{i = 1}^n(\vx(i)^2 - P)\le0}\diff \vx \notag \\
&= \prob{-\delta\le P\paren{\sum_{i = 1}^n\cN_i(0,1)^2 - n}\le0} \notag \\
&= \prob{\frac{-\delta}{P\sqrt{2n}}\le \frac{\chi^2(n) - n}{\sqrt{2n}}\le0} \notag \\
&\stackrel{n\to\infty}{\asymp} \prob{\frac{-\delta}{P\sqrt{2n}}\le \cN(0,1) \le0} \label{eqn:clt-chi-square} \\
&\stackrel{n\to\infty}{\asymp} \frac{\delta}{P\sqrt{2n}}\cdot\frac{1}{\sqrt{2\pi}} \label{eqn:rectangle} \\
&= \frac{\delta}{2P\sqrt{\pi n}}. \notag 
\end{align}
\Cref{eqn:clt-chi-square} follows since $ \frac{\chi^2(n) - n}{\sqrt{2n}} $ converges to $ \cN(0,1) $ in distribution as $ n\to\infty $. 
\Cref{eqn:rectangle} follows since the Gaussian measure of a thin interval $ \sqrbrkt{-\frac{\delta}{P\sqrt{2n}},0} $ is essentially the area of a rectangle with width $ \frac{\delta}{P\sqrt{2n}} $ and height $ \phi_P(0) = 1/\sqrt{2\pi} $ for asymptotically large $n$. 
\end{proof}
Putting \Cref{eqn:int-result} back to \Cref{eqn:int-in-paren} and applying elementary algebraic manipulation (dropping terms that are subexponential in $n$ according to \Cref{lem:bound-z}), we get the exponent $E(s,\lambda)$ of \Cref{eqn:spherical-tail}: 
\begin{align}
& -\frac{1}{n} \ln \probover{(\vbfx_1,\cdots,\vbfx_L)\sim Q^\tl}{\ol\rad^2(\vbfx_1,\cdots,\vbfx_L)\le nN} \notag \\
&\ge -\frac{1}{n}\ln \left(e^{\lambda LnN} e^{-sLPn} \sqrt{2\pi P}^{-Ln} 
 \sqrt{\frac{\pi^L}{\paren{\frac{1}{2P} - s + \lambda}^{L-1} \paren{\frac{1}{2P} - s}}}^n\right) + o(1) \notag \\
&= \underbracket{-\lambda LN + sLP + \frac{1}{2}\ln\paren{\frac{1}{2P} - s} + \frac{L-1}{2}\ln\paren{\frac{1}{2P} - s + \lambda)} + \frac{L}{2}\ln(2P)}_{\eqqcolon E(s,\lambda)} + o(1)
 . \notag 
\end{align} 
Let us maximize the above quantity first over $ \lambda\ge0 $ and then over $ s\ge0 $. 
It turns out that the maximizers are given by
\begin{align}
\lambda = \frac{L-1}{2LN} - \frac{1}{2P} + s, \quad
s = \frac{1}{2}\paren{\frac{1}{P} - \frac{1}{L(P-N)}}. 
\end{align}
By the random coding with expurgation argument as in \Cref{sec:lb-gaussian}, this gives rise to the following achievable rate
\begin{align}
\frac{1}{L-1} \max_{s\ge0}\max_{\lambda\ge0}E(s,\lambda) &=
\frac{1}{2}\ln\frac{(L-1)P}{LN} + \frac{1}{2(L-1)}\ln\frac{P}{L(P-N)}, \notag 
\end{align}
which miraculously coincides with the previous best known lower bound \Cref{eqn:compare-lb-ee} (under the standard notion of multiple packing) obtained from error exponents \cite{blinovsky-1999-list-dec-real,zhang-split-ee}. 

\section{Lower bounds via reduction to sphere packing}
\label{sec:lb-blachman-few}

\subsection{Bounded packings}
\label{sec:lb-blachman-few-bdd}
In \cite[Theorem 2]{blachman-few-1963-multiple-packing}, Blachman and Few proved another lower bound on the $ (P,N,L-1) $-list-decoding capacity. 
Their observation is that a sphere packing is going to be $(L-1)$-list-decodable if we suitably dilate the balls so that they have multiplicity of overlap at most $ L-1 $. 
The key lemma they used is the following by Few \cite[Theorem 2]{few1964multiplepacking}. 
We reproduce it here with a simpler proof and use it to get lower bounds on the $ (P,N,L-1) $-list-decoding capacity and the $ (N,L-1) $-list-decoding capacity.

\begin{lemma}
\label{lem:pack-to-multipack}
Let $ P,N>0 $. 
If $ \cC\subset\bR^n $ is a $(P,N,1)$-sphere packing, then it is also a $ (P, N', L) $-multiple packing for any $ 0\le N'<\frac{2(L-1)N}{L} $ and $ L\in\bZ_{\ge2} $. 
\end{lemma}
\begin{proof}
The proof is by contradiction. 
Suppose the conclusion of the lemma does not hold. 
Then we can find an $L$-list $ \vx_1,\cdots,\vx_{L}\in\cC $ satisfying $ \rad^2(\vx_1,\cdots,\vx_L)\le nN' $.
On the other hand,
\begin{align}
\rad^2(\vx_1,\cdots,\vx_L) 
&\ge \ol\rad^2(\vx_1,\cdots,\vx_L) 
= \frac{1}{2L^2} \sum_{(i,j)\in[L]^2:i\ne j} \normtwo{\vx_i - \vx_j}^2 
\ge \frac{1}{2L^2}L(L-1)\paren{2\sqrt{nN}}^2
= \frac{2(L-1)}{L}nN . \notag
\end{align}
where the first equality is by \Cref{eqn:avg-rad-pw-dist}. 
This is a contradiction for $ N'<\frac{2(L-1)}{L}N $ and the proof is completed. 
\end{proof}

We can combine the above lemma with the following lower bound on sphere packing density to get a lower bound on the $ (P,N,L-1) $-list-decoding capacity for any $ L\in\bZ_{\ge2} $.
It is well-known (see, e.g., \cite{blachman-1962}) that the $ (P,N,1) $-sphere packing density is lower bounded as follows. 
\begin{theorem}[\cite{blachman-1962}]
\label{thm:gv}
Let $ N,P>0 $ such that $ N\le P/2 $. 
There exist spherical codes on $ \cB^n(\sqrt{nP}) $ of minimum distance at least $ 2\sqrt{nN} $ and rate at least $ \frac{1}{2}\ln\frac{P^2}{4N(P-N)} $. 
\end{theorem}

Combining \Cref{lem:pack-to-multipack} and \Cref{thm:gv}, we immediately get:
\begin{theorem}
\label{thm:lb-blachman-few}
Let $ P,N>0 $ such that $ N\le\frac{L-1}{L}P $. 
Let $ L\in\bZ_{\ge2} $. 
The $ (P,N,L-1) $-list-decoding capacity $ C_{L-1}(P,N) $ is at least $ \frac{1}{2}\ln\frac{(L-1)^2P^2}{LN(2(L-1)P-LN)} $. 
\end{theorem}

\begin{proof}
Replace $N$ in the bound in \Cref{thm:gv} with $ \frac{LN}{2(L-1)} $. 
\end{proof}

\begin{remark}
The above bound converges to $ \frac{1}{2}\ln\frac{P^2}{N(2P-N)} $ as $ L\to\infty $ which is lower than the list-decoding capacity $ \frac{1}{2}\ln\frac{P}{N} $ for $L\to\infty$ (see \Cref{sec:listdec-cap-large}).
On the other hand, the above bound recovers the best known bound $ \frac{1}{2}\ln\frac{P^2}{4P(P-N)} $ in \Cref{thm:gv} for $L=2$. 
This is not surprising since the above bound is obtained from the bound for $L=2$. 
The above bound attains the tight values $ \infty $ and $0$ at two special points $ N/P=0 $ and $ N/P = \frac{L-1}{L} $ (see \Cref{sec:ub} for the \emph{Plotkin point}), respectively. 
\end{remark}

\subsection{Unbounded packings}
\label{sec:lb-blachman-few-unbdd}
In the proof of \Cref{lem:pack-to-multipack}, we never assumed anything about the norm of each point, therefore the lemma also holds for unbounded packings.

It is well-known that the $ (N,L-1) $-list-decoding capacity is lower bounded as follows \cite{minkowski-sphere-pack}. 
\begin{theorem}[\cite{minkowski-sphere-pack}]
\label{thm:gv-unbdd}
Let $ N>0 $. 
The $ (N,1) $-sphere packing density is lower bounded by $ \frac{1}{2}\ln\frac{1}{8\pi e N} $. 
\end{theorem}

Combining \Cref{lem:pack-to-multipack} and \Cref{thm:gv-unbdd}, we immediately get:
\begin{theorem}
\label{thm:lb-blachman-few-unbdd}
Let $ N>0 $ and $ L\in\bZ_{\ge2} $. 
The $ (N,L-1) $-list-decoding capacity $ C_{L-1}(N) $ is at least $ \frac{1}{2}\ln\frac{L-1}{4\pi eLN} $. 
\end{theorem}

\begin{remark}
The above bound converges to $ \frac{1}{2}\ln\frac{1}{4\pi eN} $ as $ L\to\infty $ which is lower than the list-decoding capacity $ \frac{1}{2}\ln\frac{1}{2\pi eN} $ for $ L\to\infty $ (see \Cref{sec:listdec-cap-large}). 
When $L=2$, the above bound recovers the best known bound $ \frac{1}{2}\ln\frac{1}{8\pi eN} $ in \Cref{thm:gv-unbdd}. 
\end{remark}

\subsection{Average-radius multiple packings}
\label{sec:lb-blachman-few-avgrad}
Finally, we comment on the average-radius version of \Cref{lem:pack-to-multipack}. 
Since the proof of \Cref{lem:pack-to-multipack} lower bounds the Chebyshev radius by average radius, the same conclusion actually also holds for the stronger notion of average-radius multiple packing. 
\begin{lemma}
\label{lem:pack-to-avgrad-multipack}
Let $ P,N>0 $. 
If $ \cC\subset\bR^n $ is a $(P,N,1)$-sphere packing, then it is also a $ (P, N', L) $-average-radius multiple packing for any $ 0\le N'<\frac{2(L-1)N}{L} $ and $ L\in\bZ_{\ge2} $. 
\end{lemma}

Since the $(L-1)$-average-radius list-decodability and the (regular) $(L-1)$-list-decodability coincide for $ L=2 $, \Cref{thm:gv,thm:gv-unbdd} also hold under the former notion. 
Combining them with \Cref{lem:pack-to-avgrad-multipack}, we can strengthen \Cref{thm:lb-blachman-few,thm:lb-blachman-few-unbdd} to the average-radius case. 

\begin{theorem}
\label{thm:lb-blachman-few-avgrad}
Let $ P,N>0 $ such that $ N\le\frac{L-1}{L}P $. 
Let $ L\in\bZ_{\ge2} $. 
The $ (P,N,L-1) $-average-radius list-decoding capacity $ \ol C_{L-1}(P,N) $ is at least $ \frac{1}{2}\ln\frac{(L-1)^2P^2}{LN(2(L-1)P-LN)} $. 
\end{theorem}

\begin{theorem}
\label{thm:lb-blachman-few-avgrad-unbdd}
Let $ N>0 $ and $ L\in\bZ_{\ge2} $. 
The $ (N,L-1) $-average-radius list-decoding capacity $ \ol C_{L-1}(N) $ is at least $ \frac{1}{2}\ln\frac{L-1}{4\pi eLN} $. 
\end{theorem}

\section{Upper bounds}
\label{sec:ub}
In this section, we prove an upper bound on the optimal density of multiple packings on the sphere. 
We first show that any $(P,N,L-1)$-multiple packing must have a constant (independent of $n$) size if $ N>\frac{L-1}{L}P $. 
We then combine this result with an Elias--Bassalygo-type (\cite{bassalygo-pit1965}) reduction to obtain an upper bound on the optimal density. 

\subsection{Plotkin bound}
\label{sec:plotkin}
In this section, we prove a Plotkin-type bound. 
It says that any $(P,N,L-1)$-list-decodable code living on a spherical cap of angular radius $ \alpha $ must have a constant size (in particular, $ R(\cC)\asymp0 $) if $ \frac{N}{P} $ is larger than the threshold $ \frac{L-1}{L}(\sin\alpha)^2 $. 
When $ \alpha = \pi $, it says that the $ (P,N,L-1) $-list-decoding capacity is zero as long as $ N>\frac{L-1}{P} $. 
Combined with lower bounds given by various constructions in the preceding sections, this gives us a sharp threshold of the largest amount of noise that a positive-rate list-decodable code can correct.

\begin{theorem}
\label{thm:plotkin}
Let $ 0<\alpha<\pi $ and $ L\in\bZ_{\ge2} $. 
Let $ P,N>0 $ be such that $ \frac{N}{P} \ge \frac{L-1}{L}(\sin\alpha)^2(1+\rho) $ for some $ \rho>0 $. 
Then any list-decodable code $ \cC $ on a spherical cap of angular radius $ \alpha $ on $ \cB^n(\sqrt{nP}) $ with list-decoding radius $ \sqrt{nN} $ and list-size at most $ L-1 $ has size at most $ |\cC|\le F(\rho^{-1} + 1) $ for some increasing function $ F(\cdot) $. 
\end{theorem}

We need some tools from Ramsey theory for the proof of the above Plotkin-type bound. 
\begin{definition}[Ramsey number]
\label{def:ramsey-number}
The \emph{$N$-colour Ramsey number} $ \sR_N(M) $ is defined as the smallest integer $K$ such that any $N$-colouring of the edges of a size-$K$ complete graph must contain a monochromatic complete subgraph of size at least $M$. 
\end{definition}

It is well-known that $ \sR_N(M) $ is finite.
There are bounds on $ \sR_N(M) $ in the literature. 
The following simple one suffices for the purposes of this section. 
\begin{theorem}[\cite{wbbj}]
\label{thm:bound-ramsey-number}
For any $ N,M\in\bZ_{\ge1} $, 
$\sR_N(M) \le 2^{N\cdot M^{N-1}}$. 
\end{theorem}

We then present a subcode extraction lemma using the above Ramsey-theoretic tools. 
\begin{lemma}
\label{lem:subcode-extraction}
For any code $ \cC\subset\cB^n(\sqrt{nP}) $, there exist $ \eps\xrightarrow{|\cC|\to\infty}0 $, an increasing function $ f(\cdot) $, a constant $ 0<d<4P $ and a subcode $ \cC'\subset\cC $ of size at least $ f(|\cC|) $ such that $ \normtwo{\vx_i - \vx_j}\in[\sqrt{n(d - \eps)},\sqrt{n(d + \eps)}] $ for all $ 1\le i\ne j\le|\cC'| $. 
\end{lemma}

\begin{proof}
The argument is somewhat standard in coding theory and we only give a proof sketch here. 
See, e.g., \cite{blinovsky-1986-ls-lb-binary} or \cite[Lemma 4, Chapter 2]{blinovsky2012book} for more detail. 
Let us build a complete graph $ \cG_\cC $ on the code $ \cC $. 
The vertices of $ \cG_\cC $ are all codewords in $ \cC $ and every pair of codewords are connected by an edge. 
Clearly, since $ \cC $ is in $ \cB^n(\sqrt{nP}) $, the distance between any pair of codewords in $ \cC $ lies in the range $ [0,2\sqrt{nP}] $. 
Let us uniformly divide the above interval into subintervals each of length $ \sqrt{2n\eps} $.
To this end, we need $ N\coloneqq\frac{4P}{2\eps} = \frac{2P}{\eps} $ many subintervals.\footnote{For simplicity, we neglect divisiblility issues and assume that $2P/\varepsilon$ is an integer.} 
Each subinterval is represented by its middle point $ \sqrt{nd} $ for some $ d\in\cS\coloneqq\curbrkt{\eps,3\eps,5\eps,\cdots,4P - \eps} $. 
If we view each $d\in\cS$ as a colour, then $\cS$ induces an edge colouring of $ \cG_\cC $ in the following sense. 
If $ \normtwo{\vx_i - \vx_j} \in [\sqrt{n(d - \eps)},\sqrt{n(d+\eps)}] $, then we say that the edge $ (\vx_i,\vx_j) $ gets colour $d$. 
Given such an $N$-colouring of the edges of $ \cG_\cC $, by Ramsey's theorem (\Cref{thm:bound-ramsey-number}), as long as the size of $ \cC $ exceeds $ \sR_N(|\cC'|) $, there must exist a monochromatic complete subgraph $ \cC'\subset\cC $. 
In other words, any code $\cC$ contains a monochromatic subcode $ \cC' $ of size at least $ \sR^{-1}_N(|\cC|) $. 
We have therefore proved the lemma by noting that monochromatic just means that all pairwise distances are $\eps$-close to $ \sqrt{nd} $ for some $ d\in\cS $. 
\end{proof}

\begin{remark}
We discuss here the $o(1)$ factor $ \eps $ and the increasing function $ f(\cdot) $ that appear in \Cref{lem:subcode-extraction}. 
Let $ K \coloneqq|\cC| $. 
Suppose $ K\ge2^{N\cdot M^{N-1}} $. 
This guarantees the existence of a size-$M$ subcode $ \cC'\subset\cC $ where 
\begin{align}
M &= \paren{\frac{\log K}{N} }^{\frac{1}{N-1}}. \notag 
\end{align}
If we take $ N = \log\log\log K $, then we can take $ f(\cdot) $ such that 
\begin{align}
M = f(K) &\coloneqq \paren{\frac{\log K}{\log\log\log K} }^{\frac{1}{\log\log\log K-1}} \xrightarrow{K\to\infty} \infty. \label{eqn:ramsey-f} 
\end{align}
Note that $ f(\cdot) $ is increasing very slowly. 
In particular, the rate of the subcode $ \cC'\subset\cC $ is vanishing even if $ \cC $ has positive rate. 
However, fortunately, this is fine for the sake of the Plotkin-type bound that we are going to prove.

Recall that $ N = \frac{2P}{\eps} $. 
Therefore we can take 
\begin{align}
\eps &= \frac{2P}{\log\log\log K} \xrightarrow{K\to\infty} 0. \notag 
\end{align}
\end{remark}

We are now ready to prove the threshold result.

\begin{proof}[Proof of \Cref{thm:plotkin}]
Let $\cC$ be an arbitrary multiple packing on a spherical cap on $ \cB^n(\sqrt{nP}) $ of angular radius $\alpha$, denoted by $ \cK_\alpha $, with list-decoding radius $ \sqrt{nN} $ and list-size at most $ L-1 $.
Since rotation does not affect list-decodability, we assume that the cap is defined as 
\begin{align}
\cK_\alpha &\coloneqq \curbrkt{\vx\in\cB^n(\sqrt{nP}):\vx(1)\ge\sqrt{nP}\cos\alpha}. \label{eqn:cap-assump} 
\end{align}

Since $ \cK_\alpha\subset\cB^n(\vv,\sqrt{nP}\sin\alpha) $ for some $ \vv\in\bR^n $, 
by \Cref{lem:subcode-extraction}, we can extract a subcode $ \cC'\subseteq\cC $ of size at least $ M\coloneqq |\cC'|\ge f(|\cC|) $ (for some increasing function $ f(\cdot) $) such that for any $ \vx_i\ne\vx_j\in\cC' $, $ \inprod{\vx_i}{\vx_j} $ almost does not depend on the choice of $ i $ and $j$. 
Specifically, there is $ \eps \xrightarrow{|\cC|\to\infty}0 $ such that $ \normtwo{\vx_i - \vx_j} \in [\sqrt{n(d - \eps)}, \sqrt{n(d + \eps)}] $ for all $ i\ne j\in[M] $ where $ 0<d<4P(\sin\alpha)^2 $ is a constant. 
Since $ \eps $ can be taken to be $ o(1) $, in the proceeding calculations, we will not carry the $o(1)$ factor for simplicity. 

Recall that if $ \cC' $ is on $ \cB^n(\sqrt{nP}) $, then for any list $ \cL\in\binom{[M]}{L} $, we have the following representations of the average squared radius
\begin{align}
\ol\rad^2\paren{\curbrkt{\vx_i}_{i\in\cL}} 
&= \frac{L-1}{L}nP - \frac{1}{L^2}\sum_{i\ne j\in\cL}\inprod{\vx_i}{\vx_j} \label{eqn:avgrad-1} \\
&= \frac{1}{2L^2} \sum_{i\ne j\in\cL} \normtwo{\vx_i - \vx_j}^2, \label{eqn:avgrad-2}
\end{align}
where $ \vx_\cL \coloneqq \frac{1}{L}\sum\limits_{i\in\cL}\vx_i $. 
The first one (\Cref{eqn:avgrad-1}) is copied from \Cref{eqn:avg-rad-spherical-formula-pw-corr}
and the second one (\Cref{eqn:avgrad-2}) is copied from \Cref{eqn:avg-rad-pw-dist}. 

Let $ \cL\in\binom{[M]}{L} $. 
Note that for any $i\in\cL$, 
\begin{align}
\normtwo{\vx_i - \vx_\cL}^2 &= \inprod{\vx_i - \vx_\cL}{\vx_i - \vx_\cL} \notag \\
&= nP - 2\inprod{\vx_i}{\vx_\cL} + \inprod{\vx_\cL}{\vx_\cL} \notag \\
&= nP - \frac{2}{L}\sum_{j\in\cL}\inprod{\vx_i}{\vx_j} + \inprod{\vx_\cL}{\vx_\cL} \notag \\
&= nP - \frac{2}{L}\paren{nP + \sum_{i\ne j\in\cL}\inprod{\vx_i}{\vx_j}} + \inprod{\vx_\cL}{\vx_\cL}. \notag
\end{align}
By the polarization identity, for any $ i\ne j\in\cL $, 
\begin{align}
\inprod{\vx_i}{\vx_j} 
&= \frac{1}{2}\paren{\normtwo{\vx_i}^2 + \normtwo{\vx_j}^2 - \normtwo{\vx_i - \vx_j}^2}
\asymp \frac{1}{2}\paren{2nP - nd}. \notag 
\end{align}
Therefore, $ \normtwo{\vx_i - \vx_\cL}^2 $ (almost) does not depend on $i$. 

Recall the following relations shown in \Cref{sec:comments-radius}: 
\begin{equation}
\begin{array}{ccccccc}
\displaystyle\exptover{\bfi\sim\cL}{\normtwo{\vx_\bfi - \vx_\cL}^2} &=& \displaystyle\min_{\vy\in\bR^n} \exptover{\bfi\sim\cL}{\normtwo{\vx_\bfi - \vy}^2} & \le & \displaystyle\min_{\vy\in\bR^n}\max_{i\in\cL}\normtwo{\vx_i - \vy}^2 & \le & \displaystyle\max_{i\in\cL}\normtwo{\vx_i - \vx_\cL}^2 \\
\verteq & && & \verteq& & \verteq \\
\ol\rad^2\paren{\curbrkt{\vx_i}_{i\in\cL}} && && \rad^2\paren{\curbrkt{\vx_i}_{i\in\cL}} && \rad^2_{\max}\paren{\curbrkt{\vx_i}_{i\in\cL}} 
\end{array} . 
\notag
\end{equation}
Since the value of $ \normtwo{\vx_i - \vx_\cL} $ is approximately the same for every $ i\in\cL $, we have $ \ol\rad^2(\curbrkt{\vx_i}_{i\in\cL})\asymp\rad^2_{\max}(\curbrkt{\vx_i}_{i\in\cL}) $. 
Since $ \rad^2(\curbrkt{\vx_i}_{i\in\cL}) $ is sandwiched between the above two quantities, we further have, for every $ \cL\in\binom{[M]}{L} $, 
\begin{align}
\ol\rad^2\paren{\curbrkt{\vx_i}_{i\in\cL}} \asymp \rad^2\paren{\curbrkt{\vx_i}_{i\in\cL}} \asymp \rad^2_{\max}\paren{\curbrkt{\vx_i}_{i\in\cL}}. 
\label{eqn:all-rad-same}
\end{align}

Moreover, we claim that the value of the above three radii (almost) do not depend on the choice of $ \cL\in\binom{[M]}{L} $. 
By \Cref{eqn:all-rad-same}, it suffices to prove this for one of the three radii. 
Indeed, for any $ \cL\in\binom{[M]}{L} $, consider the average radius:
\begin{align}
\ol\rad^2(\curbrkt{\vx_i}_{i\in\cL})
&= \frac{1}{2L^2}\sum_{i\ne j\in\cL}\normtwo{\vx_i - \vx_j} ^2 
\asymp \frac{1}{2L^2}L(L-1)nd
= \frac{L-1}{2L}nd, \label{eqn:avgrad-same}
\end{align}
independent of $ \cL $. 
The first inequality follows from the second representation of $ \ol\rad^2(\cdot) $ (\Cref{eqn:avgrad-2}). 
By the definitions of $ \rad^2_L(\cdot) $ (\Cref{eqn:rad-code}) and $ \ol\rad^2_L(\cdot) $ (\Cref{eqn:avgrad-code}), we further have $ \ol\rad^2(\cC') \asymp \rad^2(\cC') $. 

To get an upper bound on $ M=|\cC'| $, it suffices to upper bound $ \rad^2(\cC') $. 
To this end, we bound the following quantity 
\begin{align}
\exptover{\cL\sim\binom{[M]}{L}}{\ol\rad^2\paren{\curbrkt{\vx_i}_{i\in\cL}}}. \label{eqn:double-count}
\end{align}

By the first form of $ \ol\rad^2(\cdot)$ (\Cref{eqn:avgrad-1}), \Cref{eqn:double-count} equals
\begin{align}
\Cref{eqn:double-count}
&= \exptover{\cL\sim\binom{[M]}{L}}{ \frac{L-1}{L}nP - \frac{1}{L^2}\sum_{i\ne j\in\cL}\inprod{\vx_i}{\vx_j} } \notag \\
&= \frac{L-1}{L}nP - \frac{1}{L^2}\frac{1}{\binom{M}{L}} \sum_{\cL\in\binom{[M]}{L}} \sum_{i\ne j\in\cL} \inprod{\vx_i}{\vx_j} \notag \\
&= \frac{L-1}{L}nP - \frac{1}{L^2}\frac{1}{\binom{M}{L}} \sum_{i\ne j\in[M]}\sum_{\cL\in\binom{[M]}{L}}  \inprod{\vx_i}{\vx_j}\indicator{\cL\ni i,\cL\ni j} \notag \\
&= \frac{L-1}{L}nP - \frac{1}{L^2}\frac{1}{\binom{M}{L}} \sum_{i\ne j\in[M]} \inprod{\vx_i}{\vx_j}\sum_{\cL\in\binom{[M]}{L}} \indicator{\cL\ni i,\cL\ni j} \notag \\
&= \frac{L-1}{L}nP - \frac{1}{L^2}\frac{1}{\binom{M}{L}} \sum_{i\ne j\in[M]} \inprod{\vx_i}{\vx_j} \binom{M-2}{L-2} \notag \\
&= \frac{L-1}{L}nP - \frac{1}{L^2}\frac{1}{\binom{M}{L}}\paren{\sum_{(i,j)\in[M]^2}\inprod{\vx_i}{\vx_j} - \sum_{i\in[M]}\normtwo{\vx_i}^2}\binom{M-2}{L-2} \notag \\
&= \frac{L-1}{L}nP - \frac{1}{L^2}\frac{1}{\binom{M}{L}}\paren{M^2\normtwo{\ol\vx}^2 - M\cdot nP}\binom{M-2}{L-2} . \label{eqn:def-xbar}
\end{align}
In \Cref{eqn:def-xbar}, we let $ \ol\vx \coloneqq \frac{1}{M}\sum\limits_{i\in[M]}\vx_i $. 
Note that the norm of $ \ol\vx $ can be lower bounded as follows.
\begin{align}
\normtwo{\ol\vx}^2 &= 
\normtwo{\frac{1}{M}\sum_{i\in[M]}\vx_i}^2 \notag \\
&= \sum_{j = 1}^n\paren{\frac{1}{M}\sum_{i\in[M]}\vx_i(j)}^2 \notag \\
&\ge \paren{\frac{1}{M}\sum_{i\in[M]}\vx_i(1)}^2 \notag \\
&\ge \paren{\frac{1}{M}\cdot M\cdot\sqrt{nP}\cos\alpha}^2 \label[ineq]{eqn:use-cap-assump} \\
&= nP\cdot(\cos\alpha)^2. \label{eqn:bound-norm-xbar} 
\end{align}
\Cref{eqn:use-cap-assump} is by the assumption $ \cC'\subset\cC\subset\cK_\alpha $ where $ \cK_\alpha $ is defined by \Cref{eqn:cap-assump}. 
Putting \Cref{eqn:bound-norm-xbar} back to \Cref{eqn:def-xbar}, we have
\begin{align}
\Cref{eqn:double-count} &\le 
\frac{L-1}{L}nP - \frac{1}{L^2}\frac{1}{\binom{M}{L}}\paren{M^2 nP(\cos\alpha)^2 - M\cdot nP}\binom{M-2}{L-2} \notag \\
&= \frac{L-1}{L}nP - \frac{L-1}{LM(M-1)}\paren{M^2nP(\cos\alpha)^2 - MnP} \notag \\
&= \frac{L-1}{L}nP - \frac{(L-1)M}{L(M-1)}nP(\cos\alpha)^2 + \frac{L-1}{L(M-1)}nP \notag \\
&= \frac{L-1}{L}nP - \frac{L-1}{L}nP(\cos\alpha)^2 + \frac{L-1}{L(M-1)}nP - \frac{L-1}{L(M-1)}nP(\cos\alpha)^2 \notag \\
&= \frac{L-1}{L}nP(\sin\alpha)^2\paren{1 + \frac{1}{M-1}} . \label{eqn:ub-doublecount} 
\end{align}

Finally, 
\begin{align}
nN \le \rad^2(\cC') = \min_{\cL\in\binom{[M]}{L}} \rad^2\paren{\curbrkt{\vx_i}_{i\in\cL}} 
\asymp \exptover{\cL\sim\binom{[M]}{L}}{\ol\rad^2(\curbrkt{\vx_i}_{i\in\cL})}
\le \frac{L-1}{L}nP(\sin\alpha)^2\paren{1 + \frac{1}{M-1}}. \label[ineq]{eqn:to-rearrange} 
\end{align}
The asymptotic equality follows from the following two points: $(i)$ by \Cref{eqn:all-rad-same}, for any $ \cL\in\binom{[M]}{L} $, $ \rad^2(\curbrkt{\vx_i}_{i\in\cL})\asymp\ol\rad^2(\curbrkt{\vx_i}_{i\in\cL}) $; $ (ii) $ by \Cref{eqn:avgrad-same}, the value of $ \ol\rad^2(\curbrkt{\vx_i}_{i\in\cL}) $ is approximately the same for every $ \cL\in\binom{[M]}{L} $.

Assume $ \frac{N}{P} = \frac{L-1}{L}(\sin\alpha)^2(1+\rho) $ for some constant $ \rho>0 $. 
Rearranging terms on both sides of \Cref{eqn:to-rearrange}, we get 
\begin{align}
1 + \rho &\le 1+\frac{1}{M-1} &&\implies & M &\le \frac{1}{\rho} + 1. \notag  
\end{align}
Therefore we have shown $ |\cC'|\le\rho^{-1} + 1 $. 
Since $ |\cC'|\ge f(|\cC|) $, we finally get $ |\cC| \le f^{-1}(|\cC'|) \le f^{-1}(\rho^{-1}+1) = F(\rho^{-1} + 1) $, as claimed in the theorem. 
\end{proof}

\begin{remark}
In \Cref{thm:plotkin}, one can take $ F(\cdot) = f^{-1}(\cdot) $ for the $ f(\cdot) $ given by \Cref{eqn:ramsey-f}. 
Recall that $f$ is increasing very slowly. 
Since the inverse of an increasing function is also increasing, $ F $ is increasing. 
Note that $ F $ is growing as rapidly as the the Ramsey number. 
\end{remark}

\subsection{Elias--Bassalygo bound}
\label{sec:eb}
In this section, we combine the following covering lemma with \Cref{thm:plotkin} to get an upper bound on the size of an arbitrary $ (P,N,L-1) $-list-decodable code. 
\begin{lemma}[Folklore]
\label{lem:covering}
Let $ 0<\alpha<\pi $. 
There exists a collection of $ e^{n\paren{\ln\frac{1}{\sin\alpha} + o(1)}} $ spherical caps of angular radius $ \alpha $ such that they cover the whole sphere. 
\end{lemma}

\begin{remark}
A random covering enjoys the property in \Cref{lem:covering} with probability doubly exponentially close to one. 
This lemma is well-known and we omit the proof. 
In fact, the converse of \Cref{lem:covering} is also true. 
That is, any covering of the sphere by spherical caps of angular radius $\alpha$ must have size at least $ e^{n\paren{\ln\frac{1}{\sin\alpha} - o(1)}} $. 
We will not use the converse. 
\end{remark}

We combine \Cref{thm:plotkin} and \Cref{lem:covering} to prove the following upper bound on the $ (P,N,L-1) $-list-decoding capacity. 
\begin{theorem}
\label{thm:eb}
Let $ L\in\bZ_{\ge2} $ and $ N,P>0 $ such that $ \frac{N}{P}\le\frac{L-1}{L} $. 
The $ (P,N,L-1) $-list-decoding capacity $ C_{L-1}(P,N) $ is at most $ \frac{1}{2}\ln\frac{(L-1)P}{LN} $. 
\end{theorem}

\begin{remark}
\label{rk:ub-asymp}
The above bound approaches $ \frac{1}{2}\ln\frac{P}{N} $ as $ L\to\infty $. 
The latter quantity is known to be the list-decoding capacity in the large $L$ limit (see \Cref{sec:listdec-cap-large}).
In fact, it can be shown that to achieve $(P,N,L-1)$-list-decoding rate at least $ \frac{1}{2}\ln\frac{P}{N} - \eps $, the list-size has to be at least $ \Omega(\eps) $. 
To see this, note that
\begin{align}
\frac{1}{2}\ln\frac{(L-1)P}{LN} &= \frac{1}{2}\ln\frac{P}{N} - \frac{1}{2}\ln\paren{1+\frac{1}{L-1}} . \notag 
\end{align}
It then suffices to lower bound the error term:
\begin{align}
\frac{1}{2}\ln\paren{1+\frac{1}{L-1}} &\ge \frac{1}{4(L-1)} , \notag 
\end{align}
using $ \ln(1+x)\ge x/2 $ for any $ 0\le x\le1/2 $. 
This implies $ L\ge\frac{1}{4\eps} + 1 = \Omega(\frac{1}{\eps}) $. 
A logarithmic gap is noted when contrasting the above necessary condition with the sufficient condition $ L = \cO\paren{\frac{1}{\eps}\ln\frac{1}{\eps}} $ derived in \Cref{rk:lb-asymp} from \Cref{thm:lb-spherical-improved}. 

For $L=2$, the above bound becomes $ \frac{1}{2}\ln\frac{P}{2N} $ which is worse than the linear programing-type bound by Kabatiansky and Levenshtein \cite{kabatiansky-1978}. 
The above bound is tight (specifically, equals $ \infty $ and $0$, respectively) at $ N/P=0 $ and $ N/P=\frac{L-1}{L} $. 
\end{remark}

\begin{proof}
Let $ \cC\subset\cB^n(\sqrt{nP}) $ be an arbitrary $ (P,N,L-1) $-list-decodable code. 
Let $ 0<\alpha<\pi $ be such that $ \frac{N}{P} = \frac{L-1}{L}(\sin\alpha)^2(1+\rho) $ for some small constant $ \rho>0 $. 
We cover $ \cB^n(\sqrt{nP}) $ using spherical caps of angular radius $ \alpha $. 
By \Cref{thm:plotkin}, on each spherical cap, the number of codewords is at most a constant $ C_\rho = F(\rho^{-1} + 1) $, independent of $n$. 
By \Cref{lem:covering}, the number of spherical caps is $ e^{n\paren{\ln\frac{1}{\sin\alpha} + o(1)}} $. 
Therefore, the total number of codewords in $ \cC $ is at most $ C_\rho\cdot e^{n\paren{\ln\frac{1}{\sin\alpha} + o(1)}} $, meaning that the rate of $ \cC $ is at most 
\begin{align}
\ln\frac{1}{\sin\alpha} &= \frac{1}{2}\ln\frac{1}{(\sin\alpha)^2}
= \frac{1}{2}\ln\frac{(L-1)P(1+\rho)}{LN}. \notag 
\end{align}
The theorem is proved once $ \rho $ is sent to zero. 
\end{proof}

\begin{remark}
Technically, \Cref{thm:eb} is proved only for \emph{spherical} codes. 
However, by \Cref{thm:reduction-ball-to-spherical}, this is without loss of generality since any capacity-achieving code can be turned into a spherical code. 
\end{remark}

The idea in \Cref{thm:plotkin} can be adapted to the unbounded setting. 
Let $ \cC\subset\bR^n $ be an arbitrary $ (N,L-1) $-multiple packing of rate $ R(\cC) $. 
The trick is to restrict $ \cC $ to a large ball. 
Specifically, let $ \cC_P \coloneqq \cC\cap\cB^n(\sqrt{nP}) $ for a sufficiently large $ P $ such that 
\begin{align}
R(\cC_P) &\ge R(\cC) - \frac{1}{2}\ln(2\pi eP) - \delta, \label[ineq]{eqn:density-unbdd-to-bdd} 
\end{align}
for a small $ \delta>0 $. 
We divide $ \cB^n(\sqrt{nP}) $ into a sequence of shells such that the norm of vectors in each shell is in the range $ [\sqrt{n(P' - \eps)}, \sqrt{n(P' + \eps)}] $ for some $ 0\le P'\le P $ and some small $ \eps>0 $. 
By Markov's inequality (\Cref{lem:markov}), there must exist some $ P'\in[0,P] $ such that the number of codewords in the shell of radius $ \sqrt{nP'} $ is at least $ \frac{\card{\cC_P}}{P/(2\eps)} $ codewords. 
Let $ \cC_{P'}\subset\cC_P $ denote the subcode in this shell. 
The rate of $ \cC_{P'} $ is asymptotically the same as that of $ \cC_P $. 
If $ \eps $ is sufficiently small, it becomes negligible in the proof of \Cref{thm:plotkin}. 
Therefore, \Cref{thm:plotkin,thm:eb} continue to hold for $ \cC_{P'} $ and we have
$R(\cC_{P'}) \le \frac{1}{2}\ln\frac{LN}{(L-1)P'}$. 
We do not have control on $ P' $ and should take the extremal value $ P' = P $. 
By \Cref{eqn:density-unbdd-to-bdd}, we conclude
\begin{align}
R(\cC) & \le R(\cC_P) + \frac{1}{2}\ln(2\pi eP) + \delta
\asymp R(\cC_{P'}) + \frac{1}{2}\ln(2\pi eP) + \delta
\le \frac{1}{2}\ln\frac{(L-1)P}{LN} - \frac{1}{2}\ln(2\pi eP) + \delta
= \frac{1}{2}\ln\frac{L-1}{L\cdot2\pi eN} + \delta, \notag
\end{align}
provided that $ P $ is sufficiently large and consequently $ \delta $ is sufficiently small. 

We summarize our finding in the following theorem. 
\begin{theorem}
\label{thm:eb-unbdd}
Let $ N>0 $ and $ L\in\bZ_{\ge2} $. 
The $ (N,L-1) $-list-decoding capacity $ C_{L-1}(N) $ is at most $ \frac{1}{2}\ln\frac{L-1}{L\cdot2\pi eN} $. 
\end{theorem}

\begin{remark}
The above bound approaches $ \frac{1}{2}\ln\frac{1}{2\pi eN} $ as $ L\to\infty $. 
The latter quantity is known to be the list-decoding capacity in the large $L$ limit (see \Cref{sec:listdec-cap-large}). 
\end{remark}

\section{Open questions}
\label{sec:open}

We end the paper with several intriguing open questions.
\begin{enumerate}
	\item \label{itm:lp-packing}
	The problem of packing spheres in $ \ell_p $ space was also addressed in the literature \cite{rankin-sphericalcap-1955,spence-1970-lp-packing,ball-1987-lp-packing,samorodnitsky-l1}. 
	Recently, there was an exponential improvement on the optimal packing density in $ \ell_p $ space \cite{sah-2020-lp-sphere-packing} relying on the Kabatiansky--Levenshtein bound \cite{kabatiansky-1978}. 
	It is worth exploring the $ \ell_p $ version of the multiple packing problem. 
	One obstacle here is that the $ \ell_p $ average radius does not admit a closed form expression unlike the $p=2$ case.
	
	
	\item \label{itm:khinchin-beta}
	We examined the performance of multiple packings obtained from the Gaussian distribution (\Cref{sec:lb-gaussian}), the uniform distribution on a sphere (\Cref{sec:lb-spherical}) and the uniform distribution in a ball (\Cref{sec:ball-codes}), respectively. 
	It turns out that these distributions are special cases of the $n$-dimensional beta distribution $ \Beta_n(\beta) $ when $ \beta = \infty,-1,0 $, respectively. 
	If one can prove a Khinchin's inequality for beta distribution with sharp constants, then it allows us to unify our proofs for three distributions and potentially strengthen our bounds for spherical codes and ball codes. 
	Also, Khinchin's inequality for beta distribution is of independent interest for probabilists.

\end{enumerate}

			

\section{Acknowledgement}
YZ thanks Tomasz Tkocz for discussions on Khinchin's inequality with sharp constants for sums and quadratic forms. 
He also would like to thank Nir Ailon and Ely Porat for several helpful conversations throughout this project, and Alexander Barg for insightful comments on the manuscript. 

YZ has received funding from the European Union's Horizon 2020 research and innovation programme under grant agreement No 682203-ERC-[Inf-Speed-Tradeoff]. 
The work of SV was supported by a seed grant from IIT Hyderabad and the start-up research grant from the Science and Engineering Research Board, India (SRG/2020/000910). 


\bibliographystyle{alpha}
\bibliography{ref} 

\end{document}